\documentclass[12pt]{amsart}
\usepackage{mathtools}
\usepackage{amsmath,amssymb,latexsym,soul,cite,mathrsfs}
\usepackage{comment}
\usepackage{color,enumitem,graphicx}
\usepackage[colorlinks=true,urlcolor=blue,
citecolor=red,linkcolor=blue,linktocpage,pdfpagelabels,
bookmarksnumbered,bookmarksopen]{hyperref}
\usepackage[english]{babel}

\usepackage[left=2.9cm,right=2.9cm,top=2.8cm,bottom=2.8cm]{geometry}
\usepackage[hyperpageref]{backref}

\usepackage[colorinlistoftodos]{todonotes}
\makeatletter
\providecommand\@dotsep{5}
\def\listtodoname{List of Todos}
\def\listoftodos{\@starttoc{tdo}\listtodoname}
\makeatother

\usepackage[skip=5pt]{caption}
\usepackage{subcaption}

\usepackage{float} 

\usepackage{listings} 
\usepackage[utf8]{inputenc}
\usepackage[T1]{fontenc}
\usepackage{xcolor}
\usepackage{tcolorbox}
\tcbuselibrary{listings, skins, breakable}

\definecolor{codeback}{HTML}{F7F7F7}     
\definecolor{codeframe}{HTML}{F6B069}    
\definecolor{commentcolor}{HTML}{5D8AA8} 
\definecolor{keywordcolor}{HTML}{008000} 
\definecolor{stringcolor}{HTML}{B03060}  

\newtcblisting{rcode}[1][]{%
	enhanced jigsaw,
	breakable,
	arc=6pt,
	boxrule=1.5pt,
	colback=codeback,
	colframe=codeframe,
	left=15pt, right=15pt, top=5pt, bottom=5pt,
	listing only,
	listing options={%
		language=R,
		basicstyle=\ttfamily\footnotesize,
		numbers=left,
		numberstyle=\tiny\color{gray},
		numbersep=8pt,
		breaklines=true,
		breakatwhitespace=true,
		postbreak=\mbox{\textcolor{gray}{$\hookrightarrow$}\space},
		keywordstyle=\color{keywordcolor}\bfseries,
		commentstyle=\color{commentcolor}\itshape,
		stringstyle=\color{stringcolor},
		showstringspaces=false,
		morekeywords={library, TRUE, FALSE, NULL, function, if, else, for, in, while, return,
			c, list, data.frame, matrix, seq, rep, paste, sprintf, round, unique,
			length, as.data.frame, ode, ggplot, aes, geom_path, geom_point,
			labs, theme, theme_minimal, element_text, expression, scale_color_manual,
			filter, mutate, bind_rows, group_by, select, lapply, sapply}
	},
	#1
}

\numberwithin{equation}{section}

\newtheorem{theorem}{Theorem}[section]
\newtheorem{proposition}[theorem]{Proposition}
\newtheorem{lemma}[theorem]{Lemma}

\newtheorem{remark}[theorem]{Remark}

\begin{document}
	
	\title [Asymptotic behavior of damped second-order gradient systems]{Asymptotic behavior for a class of damped second-order gradient systems via Lyapunov method}
	
	\author{Renan J. S. Isneri}
	\author{Eric B. Santiago$^\star$}
	\author{Severino H. da Silva}
	
	\address[Renan J. S. Isneri]
	{\newline\indent Unidade Acad\^emica de Matem\'atica
		\newline\indent 
		Universidade Federal de Campina Grande,
		\newline\indent
		58429-970, Campina Grande - PB - Brazil}
	\email{\href{renanisneri@mat.ufcg.edu.br}{renanisneri@mat.ufcg.edu.br}}
	
	\address[Eric B. Santiago]
	{\newline\indent Unidade Acad\^emica de Matem\'atica
		\newline\indent 
		Universidade Federal de Campina Grande,
		\newline\indent
		58429-970, Campina Grande - PB - Brazil}
	\email{\href{eric.busatto@gmail.com}{eric.busatto@gmail.com}}
	
	\address[Severino H. da Silva]
	{\newline\indent Unidade Acad\^emica de Matem\'atica
		\newline\indent 
		Universidade Federal de Campina Grande,
		\newline\indent
		58429-970, Campina Grande - PB - Brazil}
	\email{\href{horacio@mat.ufcg.edu.br}{horacio@mat.ufcg.edu.br}}
	
	\pretolerance10000
	
	\begin{abstract}
		\noindent In this work we study the asymptotic behavior of a class of damped second-order gradient systems
		$$
		\ddot{u}(t) + a\dot{u}(t) + \nabla W(u(t)) = 0,
		$$
		under assumptions ensuring local convexity of the potential near equilibrium and coercivity at infinity. By introducing a Lyapunov functional adapted to the geometry of the system, we establish uniform asymptotic stability of the equilibrium for all $a \in (0,a_0]$, together with exponential decay when the potential satisfies a quadratic control near its minimum. Furthermore, complementary arguments based on semigroup theory reveal the existence of a global attractor. We also present numerical simulations for some $W$ potentials that illustrate the behavior of trajectories near equilibrium, in both dissipative and conservative regimes.
	\end{abstract}
	\thanks{$^\star$The second author was supported by grant \# 2025/469, issued in 04/01/2025, Paraíba State Research Support Foundation (FAPESQ-PB). Brazil}
	
	\maketitle
	
	\tableofcontents
	
	\setcounter{tocdepth}{1}
	

\section{Introduction}

Second-order dissipative systems of the form
\begin{equation}\label{Equation}
	\ddot{u}(t) + a \dot{u}(t) + \nabla W(u(t)) = 0,
\end{equation}
arise naturally in mechanics, optimization, and nonlinear dynamics. A classical example is the \emph{heavy-ball method with friction}, in which a particle moves on the graph of a potential $W$ under the action of inertia and a linear damping force with coefficient $a>0$. The heavy-ball method was introduced by Polyak in 1964 to speed up the convergence of iterative schemes \cite{Polyak}, and was formalized by Attouch, Goudou, and Redont \cite{AttouchGoudouRedont}, who analyzed the frictional equation to explore the minima of potential $W$. Beyond heavy-ball method with friction, the prototype equation \eqref{Equation} also models the motion of a \emph{damped simple pendulum}.  
Let $u=\theta$ be the pendulum angle, take $W(\theta)=\frac{g}{\ell}(1-\cos\theta)$, then
$$
\ddot{\theta}(t)+ a\,\dot{\theta}(t)+\dfrac{g}{\ell}\sin\theta(t)=0,
$$
which is the standard damped pendulum equation. This modeling is classical in nonlinear dynamics and the term $a\theta'$ represents the damping force, with $a > 0$. Depending on the initial energy, this system exhibits oscillatory motions, convergence to stable equilibria, and heteroclinic trajectories connecting different critical points of $W$. A comprehensive analysis of the pendulum dynamics can be found in \cite{Mawhin}. In fact, beyond these classical examples, a number of other mechanical and computational systems reduce to the general dissipative structure \eqref{Equation}, which reinforces the broad applicability and theoretical interest of the problem.

Second-order equations with constant damping also appear in the study of generalized gradient systems. Recent works have achieved significant progress in the asymptotic analysis of such systems. Jendoubi and collaborators showed that an equation of the type \eqref{Equation} can be interpreted as a quasi-gradient system via small deformations of the total energy. They proved that the desingularizing function in the Kurdyka-Łojasiewicz inequality satisfies $\varphi(t) \ge c\sqrt{t}$ when the potential $W$ is definable and of class $C^2$ \cite{Jendoubi1}. This result implies that every trajectory of a quasi-gradient system either converges to a critical point or its norm diverges to infinity. Another relevant development concerns Newton-type systems with {\it Hessian-driven damping}, combining the continuous Newton method with the heavy-ball system. Boţ and Csetnek in \cite{BotCsetnek} observed that such systems are second-order in both time and space, due to the presence of the acceleration term and the Hessian, and can be seen as a mixture of Newton’s method with the heavy-ball system. This hybrid approach inherits advantages from both methods, as shown in the optimization context by Alvarez et al. \cite{Alvarez}. We also mention that several recent works have focused on nonhomogeneous variants of \eqref{Equation} and on models where the damping parameter depends on time, such as \cite{1,2,3,4,6,7}.

The term $a\dot{u}$ in \eqref{Equation} represents a damping force, while $\nabla W(u)$ acts as a restoring force derived from the potential energy $W(u)$. In this work we focus on the second-order dissipative systems \eqref{Equation}, where $a > 0$ is a damping coefficient, $u\colon \mathbb{R} \to \mathbb{R}^N$ is the unknown trajectory, and $W\colon \mathbb{R}^N \to \mathbb{R}$ is a potential. Throughout the paper we impose the following assumptions on $W$:
\begin{itemize} 
	\item[\hypertarget{W1}{($W_1$)}] $W\in C^2(\mathbb{R}^N; \mathbb{R})$.
	
	\item[\hypertarget{W2}{($W_2$)}] There exists $\delta > 0$ and $u_* \in \mathbb{R}^{N}$ such that
	$$
	W(u_*) = 0 \quad \text{and} \quad W(u) > 0 \quad \text{for all} \quad u \in B_\delta(u_*) \setminus \{u_*\}.
	$$
	
	\item[\hypertarget{W3}{($W_3$)}] There exists $\lambda > 0$ such that
	$$
	\langle \nabla W(u), u - u_* \rangle > 0 \quad \text{for all} \quad u \in B_\lambda(u_*) \setminus \{u_*\}.
	$$
	
	\item[\hypertarget{W4}{($W_4$)}] $W$ is coercive, that is,
	$$
	\lim_{\|u\| \to +\infty} W(u) = +\infty.
	$$
	
	\item[\hypertarget{W5}{($W_5$)}] There exist constants $\alpha, \beta, \mu > 0$ such that, for every $x \in B_\lambda(u_*)$,
	$$
	\alpha\|x - u_*\|^2 \le W(x) \le \beta\|x - u_*\|^2 \quad \text{and} \quad \langle \nabla W(x), x - u_* \rangle \ge \mu\|x - u_*\|^2.
	$$
\end{itemize}
 
A simple example satisfying \hyperlink{W1}{$(W_1)$}-\hyperlink{W4}{$(W_5)$} with $u_*=0$ is the harmonic potential
$$
W(u)=\frac{1}{2}\|u\|^2.
$$
In this case, \eqref{Equation} reduces to the classical damped harmonic oscillator. This model describes a {\it mass-spring-damper system}, where $a$ represents the viscous damping and $u$ is the Hookean restoring force. Solutions exhibit decaying oscillations and converge to the equilibrium $u=0$. Such systems appear  not only in mechanics but also in electrical circuits and other dissipative models driven by an energy potential, see \cite{Edwards,Sunday1,Sunday2}. Beyond the quadratic case, the class of potentials satisfying assumptions \hyperlink{W1}{$(W_1)$}-\hyperlink{W4}{$(W_4)$} is broad and includes many models from mechanics, nonlinear dynamics, and phase-transition theory. A general and important family is obtained by taking  
$$
W(u)=\Phi(u-u_*),
$$
where $\Phi\colon\mathbb{R}^N\to\mathbb{R}$ is a convex, positive function with a global minimum at the origin. Such potentials arise naturally in the study of quasilinear Allen-Cahn type equations, see, for instance, \cite{Isneri}. In particular, coercive polynomial potentials, such as $W(u)=\|u-u_*\|^p$, with $p\ge 2$, and the exponential potential  
$$
W(u)=\frac12\big(e^{\|u-u_*\|^2}-1\big),
$$
which is convex and coercive, fits perfectly into this model. Another classical example is the Ginzburg-Landau potential  
$$
W(u)=\frac14\left(\|u\|^2-1\right)^2,
$$
whose set of global minima is the unit sphere $\{u\in\mathbb{R}^N : \|u\| = 1\}$. This potential play a central role in the theory of phase transitions and superconductivity, see \cite{Allen,5}. A further relevant example, frequently appearing in the study of dissipative hyperbolic PDEs, is the quartic double-well potential  
$$
W(u)=c_1\|u\|^{4}-c_2\|u\|^{2}+c_3,\quad c_1,c_2,c_3>0,
$$
whose structure is extensively used in nonlinear wave equations with damping. Such models constitute a natural infinite-dimensional counterpart of \eqref{Equation}, see, for instance, \cite{8}. These examples illustrate that the hypotheses imposed on $W$ encompass both linear and nonlinear restoring forces, as in the damped mass-spring system, as well as potentials arising in gradient flows, optimization, and models of phase separation.

With these examples in mind, we are ready to present our main stability theorem, which shows that the qualitative behavior of \eqref{Equation} is remarkably robust and yields uniform control of all trajectories near the equilibrium.

\begin{theorem}\label{Teo1}
	Assume that the potential $W$ satisfies \hyperlink{W1}{$(W_1)$}-\hyperlink{W4}{$(W_4)$}. Then, for every $\varepsilon > 0$ and $a_0 > 0$, there exists $\delta > 0$ such that for all $a \in (0, a_0]$, and any solution $u_a$ of \eqref{Equation} with initial data satisfying
	$$
	\|u_a(0)-u_*\| + \|\dot{u}_a(0)\| < \delta,
	$$
	the following hold:
	\begin{itemize}
		\item[\emph{(a)}] $\displaystyle \sup_{t \ge 0} \left( \|u_a(t)-u_*\| + \|\dot{u}_a(t)\| \right) \le \varepsilon$;
		\item[\emph{(b)}] $\displaystyle \lim_{t \to +\infty} u_a(t) =u_*$ and $\displaystyle\lim_{t \to +\infty} \dot{u}_a(t) = 0$.
	\end{itemize}
	If, additionally, \hyperlink{W2}{$(W_2)$} and \hyperlink{W3}{$(W_3)$} hold with $\delta = \lambda = +\infty$, then there exists $R > 0$ such that, for all $a \in (0, a_0]$, and any solution $u_a$ of \eqref{Equation},
	\begin{itemize}
		\item[\emph{(c)}] $\displaystyle \sup_{t \ge 0} \left( \|u_a(t)-u_*\| + \|\dot{u}_a(t)\| \right) \le R$;
		\item[\emph{(d)}] $\displaystyle \lim_{t \to +\infty} u_a(t) =u_*$ and  $\displaystyle\lim_{t \to +\infty} \dot{u}_a(t) = 0$.
	\end{itemize}
	Moreover, under the additional assumption \hyperlink{W5}{$(W_5)$}, for each $a >0$ there exist $\gamma(a) > 0$ and $C(a)>0$ such that the corresponding solution $u_a$ satisfies the exponential decay estimate
	\begin{equation*}
		\|u_a(t)-u_*\| + \|\dot{u}_a(t)\| \le C(a) e^{-\gamma(a) t}, \quad \forall\, t \ge 0.
	\end{equation*}
	In particular, if $0<a_1<a_2$ and $a\in[a_1,a_2]$, then there exist $C_*>0$ and $\gamma_*>0$, depending only on $a_1,a_2$ such that
	\begin{equation*}
		\|u_a(t)-u_*\| + \|\dot{u}_a(t)\| \le C_* e^{-\gamma_* t},\quad \forall\, t \ge 0,\ \forall\, a\in[a_1,a_2].
	\end{equation*}
\end{theorem}

The Theorem \ref{Teo1} asserts that the equilibrium $(u_*,0)$ of the system \eqref{Equation} is uniformly asymptotically stable with respect to the parameter $a\in(0,a_0]$: every solution starting sufficiently close to $(u_*,0)$ remains in an arbitrarily small neighborhood of the equilibrium and converges to it as $t\to +\infty$, with bounds that are uniform in $a$. The fundamental point of our study is the construction of a strict Lyapunov functional for the first-order formulation of the dissipative system
\begin{equation}\label{S-intro}
	\begin{cases}
		\dot{x}=y,\\[1mm]
		\dot{y}=-a\,y-\nabla W(x),
	\end{cases}
\end{equation}
which encodes the full dynamics of \eqref{Equation} in the phase space $\mathbb{R}^{2N}$. To this end, we introduce the Lyapunov functional
\begin{equation}
	V_a(x,y):=\frac{1}{2}\|y\|^2+2W(x)+\frac{1}{2}\|y+a(x-u_*)\|^2,
\end{equation}
which can be interpreted as a smooth perturbation of the mechanical energy
$$
E(x,y):=\frac12\|y\|^2+W(x).
$$
The functional $V_a$ allows us to obtain stability estimates that are uniform with respect to $a$. The assumptions \hyperlink{W1}{($W_1$)}-\hyperlink{W4}{($W_4$)} play a crucial role in ensuring that $V_a$ is strictly decreasing along non-trivial trajectories. The condition \hyperlink{W5}{($W_5$)} is only required to derive exponential decay of solutions, this decay cannot be uniform as $a\to 0^{+}$, since the corresponding rate $\gamma(a)$ necessarily satisfies $\gamma(a)\to 0$. Furthermore, the introduction of the Lyapunov functional $V_a$ therefore provides a robust framework for establishing the uniform asymptotic stability of \eqref{Equation}. Furthermore, we believe that $V_a$ highlights a mechanism that can be extended to non-homogeneous cases, to time-dampened systems, and to more general configurations such as trajectories over Hilbert space.

We emphasize that the conclusions of Theorem \ref{Teo1} are not restricted to a specific equilibrium. In fact, the result applies to any global minimizer $u_*$ of the potential $W$ satisfying assumptions \hyperlink{W1}{$(W_1)$}-\hyperlink{W4}{$(W_4)$}, yielding uniform asymptotic stability of the corresponding equilibrium $(u_*,0)$. Moreover, when the additional condition \hyperlink{W5}{$(W_5)$} holds, the convergence toward $(u_*,0)$ is exponential. In this sense, Theorem \ref{Teo1} provides a robust stability result around arbitrary admissible minima of the potential.

When $a=0$, the equation \eqref{Equation} reduces to the conservative system
$$
\ddot u(t) + \nabla W(u(t))=0,
$$
corresponding to the motion of a particle on $W$ without friction, so that mechanical energy is preserved. Here, $u(t)$ is the position of the particle, $\ddot{u}(t)$ its acceleration, and $\nabla W(u)$ the force derived from $W$. Under assumptions \hyperlink{W1}{($W_1$)}-\hyperlink{W4}{($W_4$)}, and in particular the coercivity of $W$, every solution remains globally bounded in $\mathbb{R}^{N}$, that is, 
$$
\sup_{t\in\mathbb{R}}\left(\|u(t)\|+\|\dot u(t)\|\right)<\infty.
$$
Thus blow-up cannot occur and no trajectory can converge asymptotically to the equilibrium. In this Hamiltonian-type regime, the combination of confinement and energy conservation leads to the typical behaviors of conservative dynamics: periodic orbits and transition solutions (heteroclinic or homoclinic). Hence, the case $a=0$ produces a qualitative change in the dynamics, yielding globally bounded but non-convergent trajectories.

In addition to the study of the asymptotic stability of \eqref{Equation}, two additional components complement our study. In Section \ref{Sec4}, we revisit the dynamics from the point of view of nonlinear semigroup theory, aiming for a comprehensive qualitative description of the dissipative flux generated by the first-order system \eqref{Equation}. This approach shows, in particular, the existence of a global attractor $\mathcal{A}$ for the corresponding semigroup, which is a compact, invariant subset of the phase space that attracts all bounded sets in forward time. Moreover, $\mathcal{A}$ contains the entire set of equilibria of the system, that is, 
$$\{(x,0)\in\mathbb{R}^{2N}:\nabla W (x)=0\}\subset \mathcal{A}.
$$ 
For readers interested in the general theory of global attractors and nonlinear semigroups, we refer to the classical monographs \cite{Livro-Alexandre,Hale,Temam}. Finally, Section \ref{Sec5} presents numerical simulations for some cases of \eqref{Equation} involving quadratic, exponential, and double-well potentials implemented in a computational environment based on \emph{R}. The simulations rely on a functional-programming framework supported by \texttt{tidyverse} (data manipulation), \texttt{deSolve} (time-integration of differential equations), and \texttt{ggplot2} (phase-plane visualization and graphical output). This computational component has a dual purpose: to geometrically illustrate the results developed in the previous sections and to reveal characteristics of the dynamics that remain inaccessible to our methods. For instance, when the potential $W$ possesses multiple global minima, such as $W(u)=(u^2-1)^2/4$, numerical evidence suggests that local maxima separating successive minima act as unstable equilibria. In the conservative regime $a=0$, simulations indicate that trajectories in the phase plane $\mathbb{R}^{2N}$ form closed periodic orbits. Furthermore, when $W(u)=u^2/2$, equation \eqref{Equation} is linear and exhibits the classical critical damping threshold at $a=2$. Our numerical experiments indicate that the same phenomenon persists in nonlinear settings: for instance, when $W(u)=(e^{u^{2}}-1)/2$, simulations suggest that $a=2$ again separates oscillatory (underdamped) dynamics from non-oscillatory (overdamped) decay. This numerical evidence points to a broader conjectural picture: nonlinear potentials with a single global minimizer may still inherit a critical damping regime analogous to the linear case. We have included complete executable \emph{R} scripts at the end of Section \ref{Sec5}. The reader can inspect, modify, or compile them to visualize trajectories, replicate phase plane geometry, and experiment with potentials or additional parameter regimes. Therefore, numerical simulation serves as a gateway to more in-depth mathematical and computational investigations.

The plan of the paper is as follows. Section \ref{Sec2} contains the proof of Theorem \ref{Teo1}, based on the construction of a strict Lyapunov functional adapted to the dissipative system. In Section \ref{Sec3} we briefly discuss the conservative case $a=0$ and the singular nature of the limit as $a\to 0^{+}$. Section \ref{Sec4} examines the global dynamics of nonlinear semigroup theory, providing a complementary viewpoint to the Lyapunov analysis of Section \ref{Sec2}. Finally, Section \ref{Sec5} presents numerical simulations for representative potentials, including quadratic, exponential, and double-well examples. 


\section{Damped System Stability}\label{Sec2}

Our main goal in this section is to prove Theorem \ref{Teo1} which provides information about the asymptotic stability of system \eqref{Equation}. Throughout this section we work under assumptions \hyperlink{W1}{$(W_1)$}-\hyperlink{W4}{$(W_4)$} on the potential $W$. The additional hypothesis \hyperlink{W5}{$(W_5)$} will be invoked explicitly only where required. Without loss of generality, we may assume that the equilibrium point $u_*$ of \eqref{Equation} is located at the origin, that is, $u_*=0$, as explained below.

\begin{remark}\label{remark-translation} \rm
To study the stability of solutions to the second-order dissipative system \eqref{Equation} we may assume, without loss of generality, that the equilibrium point is located at the origin, that is, $u_* = 0$. Indeed, if $u_* \in \mathbb{R}^N$ is a local minimum of $W$, we can define the translated potential $\widetilde{W}(u) := W(u + u_*),$ which satisfies $\widetilde{W}(0) = W(u_*)$ and inherits all the qualitative properties of $W$ near $u_*$. Now, for any solution $u$ of the original system, the function $v(t) = u(t) - u_*$ satisfies
$$
\ddot{v}(t) + a\,\dot{v}(t) + \nabla \widetilde{W}(v(t)) = 0.
$$
Consequently, the stability of the equilibrium $u=u_*$ for the original system is equivalent to the stability of the equilibrium $v=0$ for the transformed system. Thus, without loss of generality, we may assume that hypotheses \hyperlink{W1}{($W_1$)}-\hyperlink{W5}{($W_5$)} hold with the minimum of $W$ located at the origin. 
\end{remark}

We would like to clarify that, from now on, to avoid notational ambiguity, we will use $\|\cdot\|$ to denote the Euclidean norm on $\mathbb{R}^N$ with respect to the standard inner product $\langle \cdot, \cdot \rangle$, and $\|\cdot\|_*$ to denote the Euclidean norm on $\mathbb{R}^{2N}$ associated with the inner product $\langle \cdot, \cdot \rangle_*$. For a sufficiently regular solution $u$ of \eqref{Equation} defined on the time interval $I$, we consider the energy functional
\begin{equation}\label{Energy}
	E(u)(t) := \frac{1}{2} \|\dot{u}(t)\|^2 + W(u(t)),\quad t\in I.
\end{equation}
We emphasize that this expression represents the total energy of the system associated, where the term $\frac{1}{2} \|\dot{u}\|^2$ corresponds to the kinetic energy, and $W(u)$ is the potential energy. The energy functional \eqref{Energy} allows us to obtain the following result.

\begin{lemma}\label{L1}
	Every classical solution $u$ of \eqref{Equation} is globally defined.
\end{lemma}
\begin{proof}
Let $u$ be a solution defined on the maximal interval $I = (\omega_-, \omega_+)$. Deriving $E(u)$ along the solution, we get
\begin{align*}
	\frac{d}{dt}E(u)(t) &= \langle \dot{u}(t), \ddot{u}(t) \rangle + \langle \nabla W(u(t)), \dot{u}(t) \rangle \\
	&= \langle \dot{u}(t), -a \dot{u}(t) - \nabla W(u(t)) \rangle + \langle \nabla W(u(t)), \dot{u}(t) \rangle \\
	&= -a \|\dot{u}(t)\|^2 \leq 0,
\end{align*}
since $a>0$. Thus, the energy is non-increasing on $I$. Suppose, by contradiction, that $\omega_{+} < +\infty$. Fix $t_{0} \in I$. Then, for every  $t \in (t_{0},\omega_{+})$, we have
$$
E(u)(t) \leq E(u)(t_{0}).
$$
In particular, this yields
\begin{equation}\label{Eq:1}
	 W(u(t)) \leq E(u)(t_0), \quad \forall ~t \in (t_0,\omega_+).
\end{equation}
Then, by the standard continuation theorem for ordinary differential equations, there exists a sequence $(t_n) \subset (t_0,\omega_+)$ such that 
$$
t_n \to \omega_+\quad\text{and}\quad \|u(t_n)\| \to +\infty\quad\text{as}\quad n \to \infty.
$$ 
However, due to the coercivity condition \hyperlink{W4}{$(W_4)$} on $W$, we have 
$$
W(u(t_n)) \to +\infty\quad\text{as}\quad n \to \infty,
$$ 
which contradicts \eqref{Eq:1}. Therefore, $\omega_+ = +\infty$. A similar argument shows that $\omega_- = -\infty$. Hence, the solution $u$ is globally defined for all $t \in \mathbb{R}$.
\end{proof}

By introducing the variables $x=u$ and $y=\dot{u}$, the second-order dissipative system \eqref{Equation} can be rewritten as a first-order system in $\mathbb{R}^{2N}$, namely
\begin{equation}\label{S}
	\left\{\begin{array}{ll} 
		\dot{x} = y, & \\ 
		\dot{y} = -a y - \nabla W(x). & 
	\end{array}\right.
\end{equation}
This formulation is convenient for studying the qualitative behavior of solutions, since it allows the use of Lyapunov techniques. From condition \hyperlink{W2}{$(W_2)$}, we know that $(x,y) = (0,0) \in \mathbb{R}^{2N}$ is an equilibrium point of system \eqref{S}. To analyze the dynamics close to $(0,0)$, we introduce the Lyapunov functional $V_a : B_\eta(0)\times \mathbb{R}^N \to \mathbb{R}$ defined by
\begin{equation}\label{Va}
	V_a(x,y) := \frac{1}{2} \|y\|^2 + 2W(x) + \frac{1}{2} \|y + a x\|^2,
\end{equation}
where $\eta=\min\{\delta,\lambda\}$ with $\delta$ and $\lambda$  given in \hyperlink{W2}{$(W_2)$}-\hyperlink{W3}{$(W_3)$}.
	
\begin{lemma}\label{L2}
	The functional $V_a$ is a strict Lyapunov function for system \eqref{S} at the equilibrium $(0,0)$. More precisely, $V_a(0,0)=0$, and for all $(x,y)\in B_\eta(0)\times \mathbb{R}^N$ with $(x,y)\neq(0,0)$ one has
	$$
	0 < V_a(x,y)\quad\text{and}\quad\langle \nabla V_a(x,y), F_a(x,y) \rangle_* < 0,
	$$
	where $F_a(x,y)$ is the vector field of system \eqref{S}, which is given by
	\begin{equation}\label{VectorField}
		F_a(x,y) := (y, -ay - \nabla W(x)).
	\end{equation}
\end{lemma}
\begin{proof}
First, let's see that, by condition \hyperlink{W2}{$(W_2)$}, $V_a(0,0) = 0$ and, for $(x,y) \neq (0,0)$ with $(x,y)\in B_\eta(0)\times \mathbb{R}^N$, we have
$$
V_a(x,y) = \frac{1}{2} \|y\|^2 + 2W(x) + \frac{1}{2} \|y + a x\|^2 \geq W(x) > 0.
$$
Moreover, as $\nabla V_a(x,y) = \left( 2\nabla W(x) + a^2 x + a y,\; 2y + a x \right),$ we obtain
$$
\langle \nabla V_a(x,y), F_a(x,y) \rangle_* = \left\langle 2\nabla W(x) + a^2 x + a y,\; y \right\rangle + \left\langle 2y + a x,\; -a y - \nabla W(x) \right\rangle.
$$
Compute each term
$$
\left\langle 2\nabla W(x) + a^2 x + a y, y \right\rangle= 2\langle \nabla W(x), y \rangle + a^2 \langle x, y \rangle + a \|y\|^2
$$
and 
$$
\left\langle 2y + a x,-a y - \nabla W(x) \right\rangle= -2a \|y\|^2 - 2\langle y, \nabla W(x) \rangle - a^2 \langle x, y \rangle - a \langle x, \nabla W(x) \rangle.
$$
Adding these expressions, we obtain
$$
\langle \nabla V_a(x,y), F_a(x,y) \rangle_* = -a \|y\|^2 - a \langle x, \nabla W(x) \rangle.
$$
By condition \hyperlink{W3}{$(W_3)$}, $\langle \nabla W(x), x \rangle > 0$ for all $x\in  B_\eta(0)\setminus\{0\}$, and since $a > 0$, it follows that
$$
\langle \nabla V_a(x,y), F_a(x,y) \rangle_* < 0,\,  \text{ for all }(x,y)\in B_\eta(0)\times \mathbb{R}^N \text{ with }(x,y) \neq (0,0).
$$
This proves that $V_a$ is a strict Lyapunov function at the origin for the system \eqref{S}.
\end{proof}

We note that the following functional 
$$
E(x,y):=\frac12\|y\|^2 + W(x),\quad (x,y)\in B_\eta(0)\times\mathbb{R}^N,
$$
induced by the energy defined in \eqref{Energy}, acts as a Lyapunov function for the equilibrium $(0,0)$ of system \eqref{S}. However, it is not strict, since
$$
\langle \nabla E(x,y), F_a(x,y) \rangle_* =-a \|y\|^2 \le 0,\quad (x,y)\in B_\eta(0)\times \mathbb{R}^N.
$$
In order to obtain a strict Lyapunov function, it is convenient to perturb the physical energy by adding a suitable mixed term. This leads to the functional $V_a$ introduced in \eqref{Va}, which no longer has a direct mechanical interpretation, but presents interesting analytical properties for the study of equilibrium stability.

We now introduce a technical lemma that will be crucial in the analysis of uniform stability with respect to the damping parameter $a>0$.

\begin{lemma}\label{L3}
Let $m$, $a_0$ and $\kappa$ be positive real numbers. Then, the point $ (0,0) \in \mathbb{R}^{2N} $ is an interior point of the set
$$
\mathcal{W} := \left\{ (x,y) \in B_\kappa(0,0) : V_a(x,y) < m \text{ for all } a \in (0, a_0] \right\}.
$$
\end{lemma}
\begin{proof}
Firstly, we rewrite $V_a$ in an equivalent form. For $(x,y) \in B_\eta(0)\times \mathbb{R}^N$, one has
$$
V_a(x,y) = \|y\|^2 + 2W(x) + a \langle y, x \rangle + \frac{a^2}{2}  \|x\|^2.
$$
Using the Cauchy-Schwarz inequality, we deduce, for any $(x,y) \in  B_\eta(0)\times \mathbb{R}^N$ and for all $a\in (0,a_0]$, that 
\begin{equation}\label{Eq0}
	V_a(x,y)\le \|y\|^2 + 2|W(x)| + a_0 \|x\| \|y\| + \frac{a_0^2}{2} \|x\|^2.
\end{equation}
Now, since $W$ is continuous at $x = 0$ and $W(0) = 0$, given $\varepsilon >0$ satisfying $\varepsilon<\dfrac{m}{3}$, then there exists $\rho>0$ such that 
$$
\|x\| < \rho \quad  \Rightarrow\quad  |W(x)| < \frac{\varepsilon}{2}.
$$
We choose $r > 0$ small enough so that	
$$
r < \min\left\{\kappa, \rho, \sqrt{ \frac{m - \varepsilon}{1 + a_0 + \frac{a_0^2}{2} } } \right\}.
$$
Consequently, if $\|(x,y)\|_* < r$, then $\|x\| < r$ and $\|y\| < r$, and hence, by \eqref{Eq0} we obtain
$$
V_a(x,y) < r^2 \left(1 + a_0 + \dfrac{a_0^2}{2}\right) + \varepsilon,\quad \forall \, a\in (0,a_0].
$$
By choosing $r$, we conclude that 	
$$
V_a(x,y)<m \, \text{ for all }  (x,y)\in B_r(0,0)\text{ and } a\in (0,a_0].
$$
Therefore, $B_r(0,0) \subset \mathcal{W}$, that is, $(0,0) \in \operatorname{int}(\mathcal{W})$, which concludes the proof.
\end{proof}

From now on, we denote by $\varphi_a(t;z_0) = (x_a(t), y_a(t))$ the $C^1$-class solution to system \eqref{S} in $\mathbb{R}^{2N}$, associated with the parameter $a > 0$, with initial date $z_0=(x_0,y_0)$ at $t=0$, that is, $\varphi_a(0;z_0)=z_0$. Thereby, $\varphi_a(t;z_0)$ satisfies the following system of differential equations
\begin{equation*}
	\left\{
	\begin{array}{lll} 
		\dot{x}_a(t) = y_a(t), \\[4pt]
		\dot{y}_a(t) = -a\, y_a(t) - \nabla W(x_a(t)),\\
		x_a(0)=x_0,\quad y_a(0)=y_0.
	\end{array}
	\right.
\end{equation*}
Equivalently, the system can be written in vector form as
\begin{equation}\label{SF}
		\left\{
	\begin{array}{ll} 
		\dot{\varphi}_a(t;z_0) = F_a(\varphi_a(t;z_0)), \\[4pt]
		\varphi_a(0;z_0)=z_0.
	\end{array}
	\right.
\end{equation}
By Lemma \ref{L1}, we can conclude that $\varphi_a$ is a globally defined solution.

\begin{remark}\label{OBS2} \rm
If there exists $t_0\in\mathbb{R}$ such that $\varphi_a(t_0;z_0)=(0,0)$, then $\varphi_a(t;z_0)=(0,0)$ for all $t\in \mathbb{R}$. Indeed, by \hyperlink{W1}{$(W_1)$}, $F$ is locally Lipschitz, and so, we have that by the Picard-Lindelöf Theorem, the Cauchy problem
$$
\phi'(t)=F(\phi(t)),\qquad \phi(t_0)=(0,0),
$$
admits a unique solution. Since $F(0,0)=(0,0)$, the constant function $\phi\equiv(0,0)$ is a solution, by uniqueness, $\varphi_a(t;z_0)=(0,0)$ for any $t\in \mathbb{R}$.
\end{remark}

We introduce $V_a$ to capture the effect of the damping term. The intuition is: as long as the trajectory $\varphi_a(t;z_0)$ remains in a neighborhood of equilibrium (within $B_\eta(0)\times\mathbb{R}^N$), this modified energy along the trajectory cannot increase in time. We will see this precisely in the following lemma.

\begin{lemma}\label{L4}
	If $I\subset\mathbb{R}$ is an interval such that $\varphi_a(t;z_0)\in B_\eta(0)\times \mathbb{R}^N$ for all $t\in I$, then the function $t\mapsto V_a(\varphi_a(t;z_0))$ is non-increasing on $I$.
\end{lemma}
\begin{proof}
To prove the lemma, we define the following real function 
$$
h(t) = V_a(\varphi_a(t;z_0)) \quad\text{for}\quad t \in I.
$$ 
By definition, $h$ is of class $C^1$, since both $V_a$ and $\varphi_a$ are $C^1$ functions. Thus, by the Chain Rule, 
\begin{equation}\label{1eq}
	\dot{h}(t)=\frac{d}{dt}(V_a\circ\varphi_a)(t)=\sum_{i=1}^{2N}\frac{\partial V_a}{\partial x_i}(\varphi_a(t;z_0)) \dot{\varphi}_{a,i}(t;z_0),\quad \forall\,\, t\in I,
\end{equation}
where $\varphi_a(t;z_0)=(\varphi_{a,1}(t;z_0),\dots,\varphi_{a,2N}(t;z_0))$. On the other hand, 
\begin{equation}\label{2eq}
	\sum_{i=1}^{2N}\frac{\partial V_a}{\partial x_i}(\varphi_a(t;z_0)) \dot{\varphi}_{a,i}(t;z_0)=\langle\nabla V(\varphi_a(t;z_0)),\dot{\varphi}_a(t;z_0)\rangle_*,\quad \forall\,\, t\in I,
\end{equation}
and so, by \eqref{1eq} and \eqref{2eq} it follows that
$$
\dot{h}(t)=\langle\nabla V_a(\varphi_a(t;z_0)),\dot{\varphi}_a(t;z_0)\rangle_*,~~\forall\,\, t\in I.
$$
Now, since $\varphi_a$ satisfies \eqref{SF}, we obtain 
\begin{equation}\label{4eq}
	\dot{h}(t)=\langle\nabla V_a(\varphi_a(t;z_0)),F_a(\varphi_a(t;z_0))\rangle_*\,\text{ for all } t\in I.
\end{equation}
	Therefore, by Lemma \ref{L2}, $\dot{h}(t)\leq 0$ for all $t\in I$, and consequently, $h$ is a non-increasing function on $I$, which proves the result.	
\end{proof}

We are now in a position to prove the following uniform asymptotic stability with respect to the equilibrium point $(0,0)$ of system \eqref{S}.

\begin{lemma}\label{L5}
Given any $\epsilon > 0$ and $a_0>0$, there exists $\sigma > 0$ such that, for every $a\in (0,a_0]$ and for every solution $\varphi_a(t;z_0)$ of system \eqref{S}, corresponding to the damping parameter $a$, satisfying
$$
\|z_0\|_* < \sigma,
$$
we have
$$
\|\varphi_a(t;z_0)\|_* < \epsilon, \quad \text{for all } t \geq 0,\quad \text{and}\quad \lim_{t \to +\infty} \varphi_a(t;z_0) = (0,0).
$$
In other words, the asymptotic stability of the equilibrium (0,0) is uniform with respect to the damping parameter.
\end{lemma}
\begin{proof}
For each $a>0$ let $V_a:B_\eta(0)\times \mathbb{R}^N\to \mathbb{R}$ the Lyapunov function for $(0,0) \in \mathbb{R}^{2N}$ defined as in \eqref{Va}. Now, let $\epsilon > 0$ such that
$$
\overline{B_\epsilon(0,0)} = \left\{ (x,y) \in \mathbb{R}^{2N} : \|(x,y)\|_* \leq \epsilon \right\} \subset B_\eta(0)\times \mathbb{R}^N.
$$
We consider the following set 
$$
M = \left\{ V_a(x,y) : \|(x,y)\|_* = \epsilon \text{ and } a>0 \right\}.
$$
We observe that $M$ is bounded below in $\mathbb{R}$. Indeed, just note that $V_a(x,y)\ge 0$ for all $a>0$ and $(x,y)\in B_\eta(0)\times \mathbb{R}^N$. Thereby, 
$$
m := \inf M \geq 0.
$$
We claim that $m > 0$. Indeed, since 
$$
V_a(x,y)\ge W(x)\,\text{for all } a>0\text{ and }(x,y)\in M,
$$
and $W$ is continuous, as $W$ is restricted to a compact subset of $B_\eta(0)$, we conclude by \hyperlink{W1}{$(W_1)$} that $m > 0$. Now, we define the set
$$
\mathcal{W}= \left\{ (x,y) \in B_\epsilon(0,0) : V_a(x,y) < m \text{ for all }a\in (0,a_0]\right\}.
$$
From Lemma \ref{L3}, $(0,0)$ is an interior point of $\mathcal{W}$. So, there exist $\sigma  \in (0, \epsilon)$ such that $B_\sigma(0,0) \subset \mathcal{W}$.
	
\noindent\textbf{\hypertarget{C1}{Claim 1:}} If $\varphi_a(t;z_0)$ is a solution of system \eqref{S}, corresponding to the damping parameter $a\in (0,a_0]$, with initial condition $ \varphi_a(0;z_0)=z_0\in B_\sigma (0,0)$, then
$$
\|\varphi_a(t;z_0)\|_* < \epsilon \quad \text{for all } t \geq 0.
$$

To see this, note first that by Lemma \ref{L1} we can conclude that $\varphi_a$ is globally defined. Now, we assume by contradiction that Claim \hyperlink{C1}{1} is false. Then, for some $t_{0} > 0$, we must have $\varphi_a(t_{0};z_0) \notin B_\epsilon(0,0)$. Since $z_0 \in B_\epsilon(0,0)$, by continuity, there exists $t_* \in (0, t_{0})$ such that
\begin{equation}\label{3eq}
	\varphi_a(t;z_0) \in B_\epsilon(0,0)\,\text{ for all } \, t \in (0, t_*)\text{ and }\|\varphi_a(t_*;z_0)\|_* = \epsilon.
\end{equation}
Thus, $\varphi_a(t_*;z_0) \in \partial B_\epsilon(0,0)$ and hence not in $\mathcal{W}$. Consequently,
\begin{equation}\label{ineq1}
	V_a(\varphi_a(t_*;z_0)) \geq m > V_a(\varphi_a(0;z_0))=V(z_0).
\end{equation}
On the other hand, from \eqref{3eq}, we can apply Lemma \ref{L4} to ensure that
$$
V_a(\varphi_a(t_*;z_0)) \leq V_a(z_0),
$$
contradicting inequality \eqref{ineq1}. This contradiction implies that $\varphi_a(t;z_0) \in B_\epsilon(0,0)$ for all $t \geq 0$, proving Claim \hyperlink{C1}{1}.
	
\noindent\textbf{\hypertarget{C2}{Claim 2:}} If $z_0 \in B_ \sigma(0,0)$, then
$$
\lim_{t \to +\infty} \varphi_a(t;z_0) = (0,0).
$$
	
First, we observe that by Remark \ref{OBS2} we can assume without loss of generality that $\varphi_a(t;z_0) \neq (0,0)$ for any $t\ge 0$. On the other hand, from Claim \hyperlink{C1}{1}, the solution $\varphi_a(t;z_0)$ remains in $B_\epsilon(0,0)$ for all $t \geq 0$. Let us consider the function again
$$
h(t) := V_a(\varphi_a(t;z_0)), \quad t \geq 0.
$$
We have already discussed in Lemma \ref{L4} that $h \in C^1([0, \infty))$ and, because of Lemma \ref{L2}, 
$$
\dot{h}(t) = \langle \nabla V_a(\varphi_a(t;z_0)), F_a(\varphi_a(t;z_0)) \rangle_* < 0, \, \text{ for all }t > 0,
$$
from which it follows that $h$ monotonically decreasing. Moreover, $h$ is bounded below. Therefore, there exists $d \geq 0$ satisfying
$$
V_a(\varphi_a(t;z_0))\ge d \, \text{ for any }\, t\ge 0\, \text{ and }\, \lim_{t \to +\infty} V(\varphi_a(t;z_0)) = d.
$$
We claim that $d = 0$. Suppose by contradiction that $d > 0$. By continuity of $V$, and the fact that $V(0,0) = 0$, there exists $r \in (0, \epsilon)$ such that
$$
V_a(z) < d \quad \text{for all } z \in B_r(0,0).
$$
Thus, $\varphi_a(t;z_0) \notin B_r(0,0)$ for all $t \geq 0$. Since the function $f(z) = \langle \nabla V_a(z), F_a(z) \rangle_*$ is continuous and $V_a$ is strict Lyapunov functional at the origin, and the set $\overline{B_\epsilon(0,0)} \setminus B_r(0,0)$ is compact, there exists $b > 0$ such that
$$
\langle \nabla V_a(z), F_a(z) \rangle_* \leq -b \quad \text{for all } z \in \overline{B_\epsilon(0,0)} \setminus B_r(0,0).
$$
As $\varphi_a(t;z_0) \in \overline{B_\epsilon(0,0)} \setminus B_r(0,0)$, we get
$$
\dot{h}(t) \leq -b \quad \text{for all } t \geq 0.
$$
Integrating, we obtain
$$
V_a(\varphi_a(t;z_0)) - V_a(z_0) \leq -bt,\, \text{ for all }t>0,
$$
which results in
$$
V_a(\varphi_a(t;z_0)) \to -\infty \quad \text{as } t \to +\infty,
$$
which is a contradiction. Therefore, $d = 0$. To finish, take any sequence $t_n \to +\infty$. Since $\varphi_a(t_n;z_0) \in B_\epsilon(0,0)$ for any $n\in\mathbb{N}$, by Bolzano-Weierstrass, there exists a subsequence $t_{n_k}$ such that
$$
\varphi_a(t_{n_k};z_0) \to w \in \overline{B_\epsilon(0,0)}.
$$
By continuity of $V_a$,
$$
\lim_{k \to +\infty} V_a(\varphi_a(t_{n_k};z_0)) = V_a(w) = 0,
$$
which implies in $w = (0,0)$, because $V_a(z) > 0$ for all $z \neq (0,0)$. Hence,
$$
\lim_{k \to+ \infty} \varphi_a(t_{n_k};z_0) = (0,0).
$$
As every sequence $t_n \to +\infty$ has a subsequence converging to $(0,0)$, it follows that
$$
\lim_{t \to +\infty} \varphi_a(t;z_0) = (0,0),
$$
completing the proof.
\end{proof}

Let $\varphi_a(t;z_0)$ denote the solution of \eqref{SF} with initial condition $\varphi_a(0;z_0)=z_0$. The {\it domain (basin) of attraction} of the equilibrium $(0,0)$ is defined as
$$
\mathcal{A}_a(0):= \Big\{ z_0\in\mathbb{R}^{2N} \big| \lim_{t\to+\infty}\varphi_a(t;z_0)=(0,0)\Big\}.
$$
Geometrically, $\mathcal{A}_a(0)$ consists of all initial states whose trajectories converge asymptotically to the equilibrium. For any fixed $a_0>0$, the arguments developed in Lemma \ref{L5} show the existence of $\sigma>0$ such that the local inclusion
$$
B_\sigma(0,0)\subset \mathcal{A}_a(0)\quad\text{for all } a\in(0,a_0]
$$
holds. We now pass from this local information to a global analysis. When the potential $W$ has a unique global minimum and satisfies a monotonicity condition, the basin of attraction coincides with the whole phase space $\mathcal{A}_a(0) = \mathbb{R}^{2N}$.

\begin{lemma}\label{L6}
	Suppose \hyperlink{W2}{$(W_2)$} and \hyperlink{W3}{$(W_3)$} hold for $\delta =\lambda=+\infty$. The equilibrium $(0,0)$ is globally asymptotically stable, that is, $\mathcal{A}_a(0)=\mathbb{R}^{2N}$. More precisely, for each initial condition $z_0$ and $a_0>0$ there exists a radius $R>0$ such that
	$$
	\varphi_a(t;z_0)\in B_R(0,0) \text{ for all }\, t\ge0  \text{ and } a\in (0,a_0],  \text{ with } \lim_{t\to+\infty}\varphi_a(t;z_0)=(0,0).
	$$
\end{lemma}
\begin{proof}
First observe that imposing \hyperlink{W2}{$(W_2)$} and \hyperlink{W3}{$(W_3)$} with $\delta=\lambda=+\infty$ upgrades those local estimates to global ones. Therefore $V_a$ is a strict Lyapunov functional at the origin for \eqref{S} defined on $\mathbb{R}^{2N}$. Let $\varphi_a(t;z_0)$ be the solution of \eqref{SF} with initial data $z_0=(x_0,y_0)$ at $t=0$. Now, we note that there exists $L>0$ such that 
\begin{equation}\label{11eq}
	V_a(z_0)\le L\quad\text{for all }\, a\in (0,a_0).
\end{equation}
Indeed, first we remember that	
\begin{equation}\label{5eq}
	V_a(x,y)= \|y\|^2 + 2W(x) + a\langle x,y\rangle + \frac{a^2}{2}\|x\|^2,\quad \forall \, (x,y) \in \mathbb{R}^{2N},
\end{equation} 
Moreover, since
$$
\Big\| \frac{1}{\sqrt{2}}y - \frac{a\sqrt{2}}{2}x \Big\|^2\ge 0,
$$ 
expanding the square we produce the inequality 
\begin{equation}\label{7eq}
	a\langle x,y\rangle \le \frac{1}{2}\|y\|^2 + \frac{a^2}{2}\|x\|^2,\quad \forall \, (x,y) \in \mathbb{R}^{2N}.
\end{equation}
Thereby, combining the estimate \eqref{7eq} with \eqref{5eq} we obtain 
$$
V_a(z_0) \le \frac32\|y_0\|^2+2W(x_0)+a_0^2\|x_0\|^2:=L,\quad\forall\,  a\in(0,a_0],
$$
where $z_0=(x_0,y_0)$. On the other hand, 
$$
V_a(x,y) \ge \dfrac12\|y\|^2+2W(x),\quad \forall\,(x,y)\in\mathbb{R}^{2N}.
$$
Hence, using the coercivity of $W$ in \hyperlink{W4}{$(W_4)$}, we obtain
$$
V_a(x,y)\to +\infty\quad\text{as}\quad \|(x,y)\|_*\to +\infty .
$$
Consequently, there exists $R=R(L)>0$ such that $V_a(x,y)>L$ for all $(x,y)\notin B_R(0,0)$. By the monotonicity established earlier in Lemma \ref{L4} the function $t\mapsto V_a\big(\varphi_a(t;z_0)\big)$ is non-increasing, and therefore, from \eqref{11eq},
$$
V_a\big(\varphi_a(t;z_0)\big)\le V_a(z_0)\le L \quad \text{for all }\,  t\ge0\, \text{ and }\, a\in (0,a_0],
$$
and the trajectory remains in the ball $B_R(0)$ for all $t\ge0$ uniformly in $a\in (0,a_0]$, that is,
$$
\varphi_a(t;z_0)\in B_R(0,0)\,  \text{ for all }\,  t\ge0\, \text{ and }\, a\in (0,a_0].
$$
To conclude that $\varphi_a(t;z_0)\to(0,0)$ as $t\to+\infty$, it suffices to show that the origin is the only accumulation point of the trajectory. This follows exactly as in the second part of Lemma \ref{L5}. Finally, since $z_0$ is arbitrary, we conclude that $\mathcal{A}_a(0)=\mathbb{R}^{2N}$.
\end{proof}

We would like to highlight that the result above shows that the estimate is uniform with respect to the damping parameter $a\in(0,a_0]$: for each initial condition $z_0$, all trajectories $\varphi_a(t;z_0)$ remain confined in a ball of radius $R$ independent of $a$, and converge to the equilibrium. Moreover, by Lemma \ref{L5}, any solution $\varphi_a(t;z_0)$ of \eqref{S} with initial data $\varphi_a(0;z_0)=z_0$ sufficiently close to the equilibrium $(0,0)$ converges to $(0,0)$ as $t\to+\infty$. Requiring global estimates for the potential $W$, we saw in Lemma \ref{L6} that any solution $\varphi_a(t;z_0)$ of \eqref{S} converges to $(0,0)$, regardless of where the initial data $z_0$ are located. Under assumption \hyperlink{W5}{$(W_5)$} the convergence is actually exponential. Informally, \hyperlink{W5}{$(W_5)$} provides a quadratic control of the potential near the equilibrium, making $V_a$ equivalent to $\|x\|^2+\|y\|^2$ and yielding a dissipative inequality. We make this precise in the next lemma.

\begin{lemma}\label{L7}
	Assume condition \hyperlink{W5}{$(W_5)$}. Let $\varphi_a(t;z_0)=(x_a(t),y_a(t))$ be the solution of system \eqref{S} with damping parameter $a>0$. Then there exists $\sigma>0$ such that, if $\|z_0\|_*<\sigma$, the following holds:
	\begin{enumerate}
		\item[\emph{(i)}] For each fixed $a>0$, there exist constants $\gamma(a)>0$ and $C(a)>0$ such that
		$$
		\|x_a(t)\|^2+\|y_a(t)\|^2 \le C(a) e^{-\gamma(a)\,t},\quad \forall\, t\ge0.
		$$
		\item[\emph{(ii)}] If the damping parameter varies in a compact interval 
		$[a_1,a_2]$ with $0<a_1<a_2<+\infty$, then there exist $\gamma_*>0$ and $C_*>0$, depending only on $a_1,a_2$, such that the decay estimate becomes uniform
		$$
		\|x_a(t)\|^2+\|y_a(t)\|^2 \le  C_* e^{-\gamma_*\, t},\quad \forall\, t\ge0, \;\forall\, a\in[a_1,a_2].
		$$
	\end{enumerate}
\end{lemma}
\begin{proof}
Initially, we note that the following occurs
\begin{equation}\label{6eq}
	a\langle x,y\rangle \ge-\frac12\|y\|^2 - \frac{a^2}{2}\|x\|^2,\quad  \forall\,(x,y)\in\mathbb{R}^{2N}.
\end{equation}
In fact, just see that
$$
\Big\| \frac{1}{\sqrt{2}}y + \frac{a\sqrt{2}}{2}x \Big\|^2\ge 0.
$$ 
Therefore, combining the estimates \eqref{7eq} and \eqref{6eq} with \eqref{5eq} and \hyperlink{W5}{$(W_5)$} we get
\begin{equation}\label{8eq}
	m_1\big(\|x\|^2+\|y\|^2\big)\ \le\ V_a(x,y)\ \le\ m_2(a)\big(\|x\|^2+\|y\|^2\big),\quad \forall \, (x,y) \in B_\eta(0)\times \mathbb{R}^N,
\end{equation}
where 
$$
m_1=\min\Big\{\frac{1}{2},2\alpha\Big\}\quad\text{and}\quad m_2(a)=\max\Big\{\frac{3}{2},2\beta+a^2\Big\},
$$
where $\alpha$ and $\beta$ were given in \hyperlink{W5}{$(W_5)$}. Calculating the derivative of $V_a$ along of $\varphi_a$ we deduce 
\begin{equation}\label{8eqq}
	\frac{d}{dt}V_a(\varphi_a(t;z_0))=  -a\|y_a(t)\|^2 - a\langle \nabla W(x_a(t)),x_a(t)\rangle.
\end{equation}
On the other hand, given $\varepsilon=\eta$, we have by Lemma \ref{L5} that there is $\sigma>0$ such that if $\|z_0\|_*<\sigma$, then 
$$
\|\varphi_a(t;z_0)\|_* < \eta\quad\text{for all }\, t \geq 0.
$$ 
With this, we can apply \hyperlink{W1}{$(W_5)$} to inequality \eqref{8eqq} to obtain
$$
\frac{d}{dt}V_a(\varphi_a(t;z_0))\le -a\Big(\|y_a(t)\|^2+\mu\|x_a(t)\|^2\Big), \,\,  \forall \, t\ge 0,
$$
where $\mu>0$ was given in \hyperlink{W5}{$(W_5)$}. Hence, 
\begin{equation}\label{9eq}
	\frac{d}{dt}V_a(\varphi_a(t;z_0))\le -a\min\{1,\mu\}\big(\|x_a(t)\|^2+\|y_a(t)\|^2\big),\,\,  \forall \, t\ge 0.
\end{equation}
By \eqref{8eq} and \eqref{9eq}, 
$$
\frac{d}{dt}V_a(\varphi_a(t;z_0)) \le  -\gamma(a)\, V_a(\varphi_a(t;z_0)),\quad\text{for all }\, t\ge 0,
$$
where 
$$
\gamma(a):=\frac{a\min\{1,\mu\}}{m_2(a)}. 
$$
Now we consider the function $g(t):=e^{\gamma(a) t}V_a(\varphi_a(t;z_0))$ for $t\ge 0$. So, 
$$
\dot{g}(t)=e^{\gamma(a) t}\left(\frac{d}{dt}V_a(\varphi_a(t;z_0))+\gamma(a) V_a(\varphi_a(t;z_0))\right)\le 0.
$$
Therefore $g$ is non-increasing on $[0,+\infty)$, and thus, 
\begin{equation}\label{10eq}
	V_a(\varphi_a(t;z_0))\le V_a(z_0)e^{-\gamma(a) t}.
\end{equation}
From \eqref{10eq} and using the lower bound in \eqref{8eq}, one has
$$
\|x_a(t)\|^2+\|y_a(t)\|^2 \le \frac{V_a(z_0)}{m_1}e^{-\gamma(a) t}\, \text{ for all }t\ge 0,
$$ 
where $C(a)=\frac{V_a(z_0)}{m_1}$. This proves item (i). To prove the uniform estimate (ii), assume $a\in[a_1,a_2]$ with 
$0<a_1<a_2$. Then $m_2(a)\le m_2(a_2),$ and hence
$$
\gamma(a)= \frac{a\min\{1,\mu\}}{m_2(a)}\ge \frac{a_1\min\{1,\mu\}}{m_2(a_2)}=: \gamma_* > 0.
$$In addition, by \eqref{11eq} there exists a constant $L>0$, independent of $a\in(0,a_2]$, such that $V_a(z_0)\le L$ for all such $a$. Combining these facts, we obtain the uniform exponential bound with $C_* = L/m_1$. This concludes the proof of the lemma.
\end{proof}

It is worth noting that the exponential decay obtained in Lemma \ref{L7} cannot be uniformized when $a\to 0^{+}$. In fact, $\gamma(a)\to 0$ as $a\to 0^{+}$, and this degeneracy reflects an intrinsic dynamic characteristic of the system. When $a$ is small, the system \eqref{S} approaches the conservative equation $\ddot{u}+\nabla W(u)=0$, which has no decay mechanism (See Section \ref{Sec3}). This phenomenon also has a geometric interpretation: for each fixed $a>0$, the Lyapunov function $V_a$ defines a family of ellipsoidal level sets that control the Euclidean norm of $(x,y)$, but these sets become distorted as $a$ varies and are not uniformly comparable to each other.

We have finished this section by providing proof of our main result.

\begin{proof}[{\bf Proof of Theorem \ref{Teo1}}]
The proof follows by combining the results established in this section. Local stability and convergence in parts (a)-(b) follow from Lemma \ref{L5}, while the global estimates in (c)-(d) are obtained in Lemma \ref{L6}. Finally, the exponential decay stated in the last part of the theorem is a direct consequence of the estimates derived in Lemma \ref{L7}. This completes the proof.
\end{proof}

\section{The Conservative Case $a=0$}\label{Sec3}

In this section we discuss the dynamics of the conservative system
\begin{equation}\label{EqConservative}
	\ddot{u}(t) + \nabla W(u(t)) = 0,
\end{equation}
which arises as the limit of \eqref{Equation} when the damping parameter $a$ tends to zero. Throughout this section, we continue to assume that $W$ satisfies \hyperlink{W1}{$(W_1)$}-\hyperlink{W4}{$(W_4)$} and, without loss of generality, that the global minimum of $W$ is located at $u_*=0$. In contrast with the dissipative
case, the conservative dynamics do not generate asymptotic stability, since no decay mechanism is present. To analyze the qualitative behavior of \eqref{EqConservative}, it is convenient to rewrite it as a first-order system in $\mathbb{R}^{2N}$ by setting $x=u$ and $y=\dot{u}$:
\begin{equation}\label{S0}
	\begin{cases}
		\dot{x} = y,\\[2mm]
		\dot{y} = -\nabla W(x).
	\end{cases}
\end{equation}
The natural energy associated with \eqref{EqConservative},
$$
E(x,y) := \frac12\|y\|^2 + W(x), \quad (x,y)\in B_\eta(0)\times \mathbb{R}^N,
$$
where $\eta=\min\{\delta,\lambda\}$ with $\delta$ and $\lambda$ given in \hyperlink{W2}{$(W_2)$}-\hyperlink{W3}{$(W_3)$}, acts as a Lyapunov function for the equilibrium $(0,0)$. Indeed, if $(x(t),y(t))$ is a solution of \eqref{S0}, then deriving 
$E(x(t),y(t))$ along the trajectory yields
$$
\frac{d}{dt}E(x(t),y(t))= \langle y(t),\dot{y}(t)\rangle + \langle \nabla W(x(t)),\dot{x}(t)\rangle.
$$
Using $\dot{x}=y$ and $\dot{y}=-\nabla W(x)$, we immediately obtain
$$
\frac{d}{dt}E(x(t),y(t))= -\langle y(t),\nabla W(x(t))\rangle+ \langle \nabla W(x(t)),y(t)\rangle= 0.
$$
Thus, if $u(t)$ is any solution of \eqref{EqConservative}, then the energy is conserved in time, that is,
\begin{equation}\label{EnergyConservation}
	\frac12\|\dot{u}(t)\|^2 + W(u(t)) = \tau,\quad \forall\, t\in\mathbb{R},
\end{equation}
for some $\tau\ge 0$ determined by the initial condition. From the Lyapunov Criterion, conservation of energy implies that 
$(0,0)$ is stable: for every $\varepsilon>0$ there exists 
$\delta>0$ such that
$$
\|u(0)\| + \|\dot{u}(0)\| < \delta\quad\Longrightarrow\quad\sup_{t\ge 0} \big( \|u(t)\| + \|\dot{u}(t)\| \big) \le \varepsilon.
$$
However, stability cannot be asymptotic. Indeed, if a nontrivial solution $(u(t),\dot{u}(t))$ were to converge to $(0,0)$ as $t\to+\infty$, then \eqref{EnergyConservation} would force $\tau=0$, which implies that $u(t)\equiv 0$ identically. Hence the only trajectory converging to the equilibrium is the constant solution.

The conservation law \eqref{EnergyConservation} also clarifies the qualitative nature of conservative trajectories. For $\tau>W(0)=0$, every solution of \eqref{EqConservative} lies on the energy level surface
$$
\mathcal{F}_\tau:= \Big\{(x,y)\in B_\eta(0)\times \mathbb{R}^N : \|y\|^2 + 2W(x) = 2\tau\Big\}.
$$
Since $W$ is coercive by assumption \hyperlink{W4}{$(W_4)$}, every sublevel set $\{x\in B_\eta(0) : W(x)\le \tau\}$ is compact. Hence, for any fixed $\tau>0$, the identity \eqref{EnergyConservation} implies that
$$
u(t)\in \{x\in B_\eta(0) : W(x)\le\tau\},\quad \forall\, t\in\mathbb{R},
$$
and therefore $u(t)$ is globally bounded. Furthermore, the velocity also satisfies
$$
\|\dot{u}(t)\|^2 = 2(\tau - W(u(t))) \le 2\tau,
$$
which shows that $\dot{u}(t)$ is also uniformly bounded for all $t\in\mathbb{R}$. Consequently, every solution of the conservative system is globally bounded in phase space:
$$
\sup_{t\in\mathbb{R}}\Big( \|u(t)\| + \|\dot{u}(t)\| \Big)<+\infty.
$$
Therefore, the conservative trajectories generated by \eqref{EqConservative} remain confined to a compact subset of $\mathbb{R}^{2N}$ and cannot satisfy $\|u(t)\|\to+\infty$ as $t\to\pm \infty$. Instead, qualitative behavior is typically oscillatory or transitional.

\subsection{A singular limit as $a\to 0^{+}$}

We now turn to the relation between the conservative regime and the dissipative dynamics studied in Section \ref{Sec2}. Assume that $W$ satisfies \hyperlink{W1}{$(W_1)$}-\hyperlink{W5}{$(W_5)$}. The conservative system \eqref{EqConservative} arises as the limit of the dissipative equation \eqref{Equation} when the damping parameter $a$ tends to zero. In the dissipative regime, Section \ref{Sec2} established that solutions are asymptotically stable and even converge exponentially fast to the equilibrium with decay rate $\gamma(a)>0$ depending on $a$. A natural question is whether this behavior persists, in some sense, in the limit as $a\to 0^{+}$. The answer is negative: the limit dynamics is not asymptotically stable, and the passage $a\to 0^{+}$ is singular. This phenomenon can be described rigorously through a perturbation argument. Let $(a_n)$ be a sequence of positive numbers with $a_n \to 0$, and let $u_n$ be the solution of the perturbed problem
\begin{equation}\label{EqDamped-n}
	\ddot u_n(t) + a_n \dot u_n(t) + \nabla W(u_n(t)) = 0,
\end{equation}
with fixed initial data $(u_n(0),\dot u_n(0)) = (z_0,z_1)$, where $(z_0,z_1)\neq(0,0)$ and $\|(z_0,z_1)\|<\sigma$. Here $\sigma>0$ depends on a prescribed $\varepsilon\in(0,\delta)$, with $\delta>0$ given in \hyperlink{W2}{\((W_2)\)}, and is provided by the uniform stability result of Lemma \ref{L5}. In particular, Lemma \ref{L5} ensures that
$$
\lim_{t\to+\infty} (u_n(t),\dot u_n(t)) = (0,0)\quad\text{and}\quad\|u_n(t)\| + \|\dot u_n(t)\| \le \varepsilon \quad\text{for all } t\ge 0\text{ and }n\in\mathbb{N}.
$$
Moreover, by Lemma \ref{L7}, each $u_n$ satisfies the exponential decay estimate
\begin{equation}\label{Exp}
	\|u_n(t)\| + \|\dot u_n(t)\|\le C e^{-\gamma(a_n) t}, \qquad t\ge 0,
\end{equation}
where $C>0$ is a constant independent of $a_n$. A compactness argument shows that, up to a subsequence, $u_n$ converges in $C^2_{\mathrm{loc}}([0,+\infty))$ to a limit
function $u_0$, which solves the conservative system
$$
\ddot u_0 + \nabla W(u_0)=0
$$
with initial data $u_0(0)=z_0$ and $\dot u_0(0)=z_1$, and satisfying
\begin{equation*}
	\sup_{t\ge 0}\left(\|u_0(t)\|+\|\dot{u}_0(t)\|\right) \leq \epsilon.
\end{equation*}
Although each $u_n$ converges asymptotically to the equilibrium, the limit solution $u_0$ does not. Indeed, suppose by contradiction that
$$
\lim_{t\to+\infty} u_0(t)=0\quad\text{and}\quad
\lim_{t\to+\infty} \dot u_0(t)=0.
$$
Since $W$ is continuous and $W(0)=0$, it follows that $W(u_0(t))\to 0$ and $\|\dot u_0(t)\|\to 0$. Inserting these limits into the conservation law \eqref{EnergyConservation}, we obtain $\tau=0$. But $\tau=0$ implies
$$
\dot u_0(t)\equiv 0\quad\text{and}\quad W(u_0(t))\equiv 0,
$$
so that $u_0(t)= 0$ for all $t\in\mathbb{R}$. This contradicts the fact that $(u_0(0),\dot u_0(0))\neq(0,0)$. Therefore,
$$
\lim_{t \to +\infty} (u_0(t),\dot{u}_0(t)) \neq (0,0).
$$
This argument shows that although each damped solution converges to the equilibrium, the limit solution does not. Equivalently, the exponential decay rate $\gamma(a)$ obtained for $a>0$ must necessarily satisfy $\gamma(a)\to 0$ as $a\to 0^+$, and no uniform decay estimate can hold on an interval of the form $(0,a_0]$. The transition from $a>0$ to $a=0$ represents a loss of asymptotic stability and a qualitative change in the global dynamics.


\section{Existence of Attractors}\label{Sec4}

In this section, we investigate the existence of a global attractor for the semigroup generated by system \eqref{S}. Our aim is to obtain a global qualitative description of the asymptotic dynamics of the system, complementing the pointwise Lyapunov analysis of Section \ref{Sec2}. We recall that, by Lemma \ref{L1}, for every $a>0$ and every initial condition $z_0 = (x_0,y_0)\in\mathbb{R}^{2N}$, the Cauchy problem \eqref{SF} associated with the dissipative system \eqref{Equation} admits a unique global solution
$$
\varphi_a(t;z_0) = (x_a(t),y_a(t)) \in \mathbb{R}^{2N}, \quad t\ge 0.
$$
A nonlinear semigroup on $\mathbb{R}^{2N}$ is a family of mappings $\{T(t)\}_{t\ge 0}$ given by $T(t):\mathbb{R}^{2N}\to\mathbb{R}^{2N}$, such that
\begin{itemize}
	\item[(a)] $T(0):\mathbb{R}^{2N}\mapsto\mathbb{R}^{2N}$ is such that $T(0)z=z$;
	\item[(b)] $T(t+s)=T(t)\circ T(s)$ for all $t,s\ge 0$;
	\item[(c)] $(t,z_0)\mapsto T(t)z_0$ is continuous from $[0,\infty)\times \mathbb{R}^{2N}$ into $\mathbb{R}^{2N}$. 
\end{itemize}
Using the well-defined \eqref{SF}, we define the semigroup naturally associated with the system \eqref{S} by
\begin{equation}\label{semigroup-definition}
	T(t) z_0 \coloneqq \varphi_a(t;z_0), \quad t\ge 0.
\end{equation}
This semigroup encodes the full dynamical evolution of solutions of  \eqref{S} under the flow generated by the parameter $a>0$. For further background on semigroup theory and its applications to dynamical systems, we refer the reader to \cite{Livro-Alexandre} and \cite{Hale}.

For our purposes, following the definition in \eqref{Energy}, the energy of the solution $\varphi_a(t;z_0) = (x_a(t),y_a(t))$ is given by
$$
E(\varphi_a(t;z_0)) \coloneqq \frac12 \|y_a(t)\|^2 + W(x_a(t)),\quad t\ge 0.
$$
As shown in Lemma~\ref{L1}, the energy is non-increasing along trajectories, since $\dfrac{d}{dt} E(\varphi_a(t;z_0)) \le 0$ for all $t\ge 0$. Therefore,
\begin{equation}\label{decreasing-property}
	E(\varphi_a(t;z_0))\le E(z_0),\quad \forall\, t\ge 0.
\end{equation}
As emphasized earlier in Remark \ref{remark-translation}, throughout this section we continue to assume that the global minimizer $u_*$ of $W$ in is located at the origin, i.e., $u_*=0$, and that the potential $W$ satisfies assumptions \hyperlink{W1}{$(W_1)$}-\hyperlink{W4}{$(W_4)$}.

\begin{lemma}\label{asd}
	Let $B\subset \mathbb R^{2N}$ be any bounded set containing the initial data $z_0$.
	Then there exists a bounded set $K\subset \mathbb{R}^{2N}$ such that 
	$$
	T(t) B \subset K,\quad \forall\, t\ge 0.
	$$
	In other words, $K$ serves as a absorbing set for the family of semigroup $\{T(t)\}_{t\ge 0}$.
\end{lemma}
\begin{proof}
Fix $a>0$ and let $z_0=(x_0,y_0)\in B$. Since $B$ is bounded in $\mathbb{R}^{2N}$ and $W$ is continuous by \hyperlink{W1}{$(W_1)$}, there exists a $R_B>0$ such that
$$
\sup_{z_0\in B} E(z_0)\le R_B.
$$
Therefore, by the monotonicity property of the energy in \eqref{decreasing-property}, we obtain
$$
W(x_a(t))\leq R_B\quad\text{and}\quad\|y_a(t)\|^2\leq 2R_B\quad\forall \, t\geq 0,
$$
where $x_a$ and $y_a$ denote the components of the trajectory $\varphi_a(t;z_0)$. Next, by the coercivity assumption \hyperlink{W4}{$(W_4)$}, the sublevel set
$$
\mathcal{S}_{B} \coloneqq \{ x \in \mathbb{R}^{N}:W(x)\le R_B\}
$$
is bounded. Consequently, we may define the positive real number
$$
C_B \coloneqq \sup \left\{\|x\|^{2} : x \in \mathcal{S}_{B}\right\} < \infty.
$$
In this way, we have 
$$
\Vert T(t)z_0 \Vert_{\ast}^2 = \Vert (x_a(t), y_a(t))\Vert_{\ast}^2 = \Vert x_a(t) \Vert^2 + \Vert y_a(t) \Vert^2 \leq C_B + 2R_B,\quad \forall\, t\ge 0, 
$$
Now, defining
$$
K\coloneqq\Big\{(x,y)\in\mathbb{R}^{2N}:\|(x,y)\|_{\ast}^{2}\le \sqrt{C_B + 2R_B}\Big\}
$$
it follows the inclusion $T(t)B \subset K$ for all $t\geq 0$, which finishes our proof.
\end{proof}

The inclusion $T(t)B \subset K$ for all $t\ge 0$, means precisely that every trajectory starting in $B$ remains bounded for all time, with a bound that is independent of the parameter $a>0$. Thus, the set $K$ acts as a absorbing set for the family of dissipative systems \eqref{Equation}.

\begin{remark}
	Although the semigroup in \eqref{semigroup-definition} is denoted simply by $\{T(t)\}_{t\ge 0}$, it is implicitly associated with the dissipative system \eqref{Equation} for a given damping parameter $a>0$. A crucial point in the proof of Lemma \ref{asd} is that the absorbing set $K$ can be chosen independently of the parameter $a>0$. In other words, the same bounded set $K\subset\mathbb{R}^{2N}$ absorbs the image of any bounded set $B$ under the semigroup generated by \eqref{Equation}, uniformly with respect to $a>0$. 
\end{remark}

For completeness, we recall the notion of attractor used in this context. A set $\mathcal{A}\subset \mathbb{R}^{2N}$ is called the \emph{global attractor} for the semigroup $\{T(t)\}_{t\ge 0}$ if $\mathcal{A}$ is compact, $\mathcal{A}$ is invariant, that is, $T(t)\mathcal{A}=\mathcal{A}$ for all $t\ge 0$ and $\mathcal{A}$ attracts every bounded subset of $\mathbb{R}^{2N}$, that is, for every bounded $B\subset \mathbb{R}^{2N}$,
$$
\lim_{t\to\infty}\operatorname{dist}(T(t)B,\mathcal{A})=0,
$$
where $\operatorname{dist}(T(t)B,\mathcal{A})$ denotes the distance between sets $T(t)B$ and $\mathcal{A}$. This definition ensures that the global attractor is the minimal compact set that attracts all bounded sets. Next, we will show how the absorbing set $K$ constructed in Lemma \ref{asd} naturally leads to the existence of the global attractor for $\{T(t)\}_{t\ge 0}$.

\begin{proposition}
The semigroup $\{T(t)\}_{t\ge 0}$ defined by \eqref{semigroup-definition} admits a global attractor $\mathcal{A} \subset \mathbb{R}^{2N}$. Moreover, the attractor can be characterized as
$$
\mathcal{A} = \mathcal{U}^u(\mathcal{E})\coloneqq\{z\in \mathbb{R}^{2N}:T(-t)z\to \mathcal{E}\text{ as }t\to +\infty\},
$$
where $\mathcal{E}:=\{(x,0)\in\mathbb{R}^{2N}:\nabla W (x)=0\}$ denotes the set of equilibria of system \eqref{S}. Furthermore, if the potential $W$ admits a unique global minimizer $u_*$, then $\mathcal{A}= \{(u_*,0)\}.$
\end{proposition}
\begin{proof}
From \cite[Theorem 1.1]{Temam} or \cite[Corollary 2.13]{Livro-Alexandre}, a necessary and sufficient condition to obtain a global attractor is to prove the existence of a compact absorbing set. Moreover, in this case, the attractor is given by the $\omega$-limit set of such absorbing set. Since we are in a finite-dimensional setting, the set $K$ from Lemma  \ref{asd} fulfill the requirements. Therefore, $\mathcal{A} \coloneqq \omega (K)$ is the global attractor for the semigroup \eqref{semigroup-definition}. Combining Lemma \ref{L2} with Theorem 3.8.5 of \cite{Hale}, we conclude that $\mathcal{A} = \mathcal{U}^u(\mathcal{E})$. Finally, if the potential $W$ admits a unique global minimizer $u_*$, then the equilibrium set reduces to the single point $\mathcal{E}=\{(u_*,0)\}$. In this case, the existence of a strict Lyapunov functional, as established in Lemma \ref{L2}, this implies that no non-trivial invariant set can persist far from equilibrium. Consequently, $\mathcal{A}=\{(u_*,0)\}$.
\end{proof}

Since the potential $W$ is coercive, from \hyperlink{W4}{$(W_4)$}, the set of equilibria $\mathcal{E}$ is bounded in $\mathbb{R}^{2N}$. Consequently, 
$$
\mathcal{E} \subset \mathcal{A}.
$$
In particular, for the Ginzburg-Landau potential
$$
W(u)=\frac{1}{4}\big(\|u\|^2-1\big)^2,
$$
the set of equilibria coincides with the unit sphere, $\mathcal{E}=\{(u,0)\in\mathbb{R}^{2N} : \|u\|=1\}$. In this case, the global attractor $\mathcal{A}$ contains a nontrivial continuum of equilibrium points.


\section{Numerical Simulations}\label{Sec5}

The purpose of this section is to present a numerical study that illustrates the dynamics of the damped system \eqref{Equation}. Our approach exploit the functional programming in R language, supported by the following libraries:

\begin{itemize} 
\item[(i)] \href{https://www.tidyverse.org/}{tidyverse} for data treatment; 
\item[(ii)] \href{https://cran.r-project.org/web/packages/deSolve/index.html}{deSolve} for numerical simulations with differential equations; 
\item[(iii)] \href{https://ggplot2.tidyverse.org/}{ggplot2}   for plotting and graphic visualizations. 
\end{itemize}
We will use this numerical framework to investigate the behavior of solutions associated with different types of potentials. In particular, we will consider the quadratic potential, which yields a damped linear dynamics, and the double-well potential, widely known in the literature as the Ginzburg-Landau potential, whose nonlinearity introduces more complex phenomena such as multiple equilibria and phase transitions. To complement these two classical examples, we also examine an exponential potential, whose superquadratic growth produces a strongly nonlinear dynamic. Throughout the section, the analysis will focus on the one-dimensional setting.


\subsection{Quadratic Potential}

We begin with the simplest potential satisfying conditions \hyperlink{W1}{$(W_1)$}-\hyperlink{W5}{$(W_5)$}, namely
$$
W(u)=\frac{1}{2}u^{2},\quad u\in\mathbb{R}.
$$
In this case, equation \eqref{Equation} reduces to
\begin{equation}\label{QP}
	\ddot{u}(t)+a\dot{u}(t)+u(t)=0,
\end{equation}
the classical damped harmonic oscillator. From a physical viewpoint, this system models small oscillations of a mass attached to a linear spring under viscous damping. Its dynamics are completely determined by the characteristic equation
$$
\lambda^{2}+a\lambda+1=0,
$$
whose roots describe the underdamped ($a<2$), critically damped ($a=2$), and overdamped regimes ($a>2$). Although the transient behavior varies substantially across these regimes, the presence of damping ensures that, in all cases, the solution converges to the equilibrium $u=0$ as $t\to +\infty$. We rewrite this second-order ODE \eqref{QP} as a first-order system by setting $y_1 = u$ and $y_2 = \dot{u},$ which yields  
$$
\begin{cases}
	\dot{y}_1 = y_2, \\
	\dot{y}_2 = -a y_2 - y_1.
\end{cases}
$$
The system above is then integrated using methods from the \textsf{deSolve} package in \textsf{R} language. 

\medskip
\begin{center}
\includegraphics[width=\textwidth]{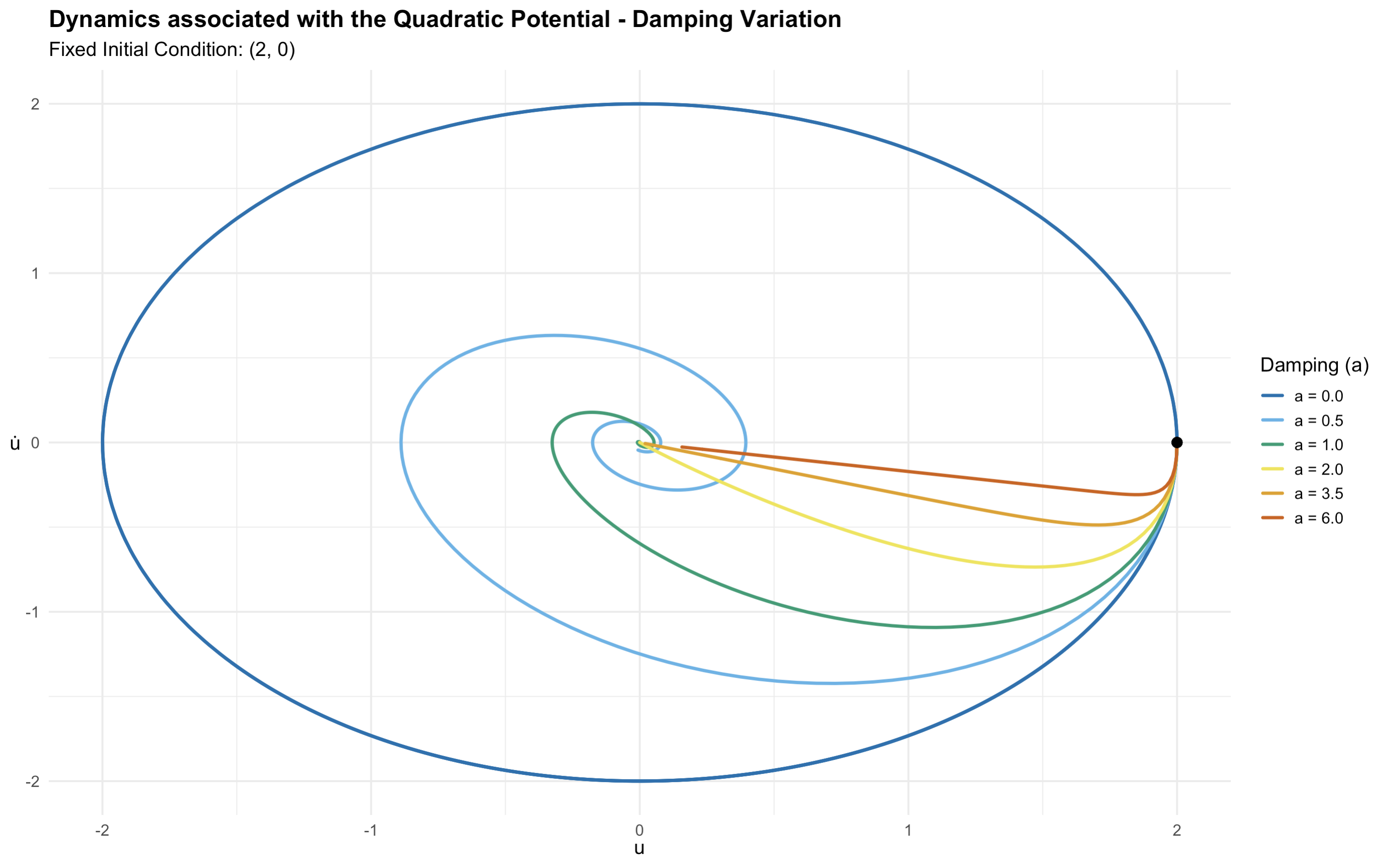}
\end{center}

The first graph illustrates the sensitivity of the solution with respect to the damping parameter $a$ 
for the quadratic potential $W(u) = u^2/2$. 
We fix the initial condition at $(2, 0)$ and vary $a$ across six different values 
covering the distinct dynamical regimes of the damped harmonic oscillator. 

\begin{itemize}
\item[(1)] Conservative Limit ($a=0$): The trajectory forms a closed orbit (an ellipse), representing 
periodic motion with constant energy. 
	
\item[(2)] Underdamped Regime ($0 < a < 2$): For $a=0.5$ and $a=1.0$, 
the trajectories spiral towards the origin, 
causing the solution to oscillate multiple times before decaying. 
	
\item[(3)] Critically Damped Regime ($a=2$): The trajectory goes to the equilibrium 
as quickly as possible. 
	
\item[(4)] Overdamped Regime ($a > 2$): For $a=3.5$ and $a=6.0$, 
the trajectories approach the origin monotonically. 
\end{itemize}

\medskip
\begin{center}
\includegraphics[width=\textwidth]{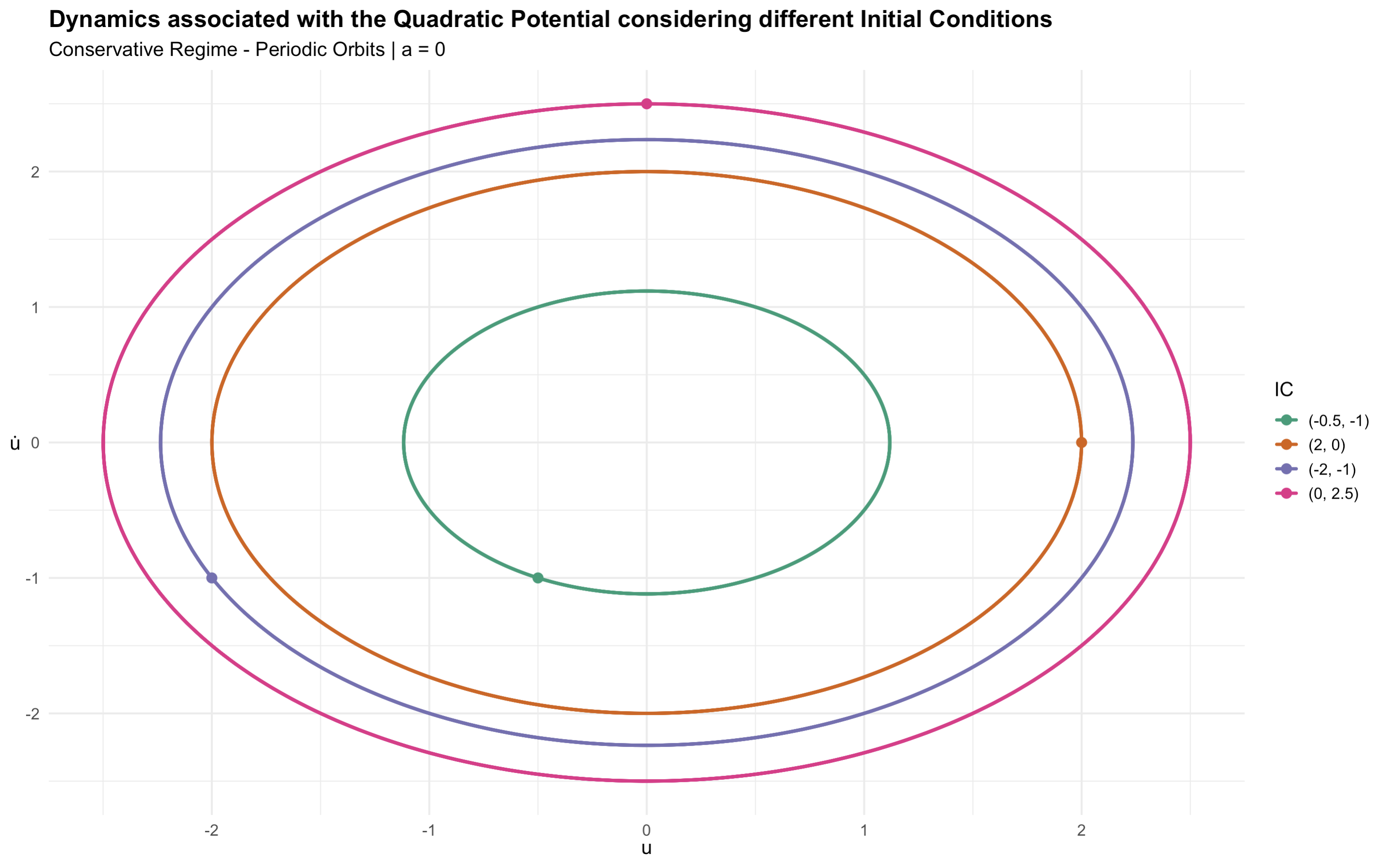}
\end{center}
\medskip

The second graph displays the dynamics for $a=0$ (conservative case).  
As we can see, the phase curves in this case are periodic orbits in the form of 
concentric ellipses centered at the origin. In this case $a=0$, the equation \eqref{QP} reduces to the undamped harmonic oscillator
$$
\ddot{u}(t)+u(t)=0,
$$
whose general solutions
$$
u(t)=c_1\sin(t)+c_2\cos(t)
$$
are periodic and conserve mechanical energy, as studied in Section \ref{Sec3}.


\subsection{Double-Well Potential}

The double-well potential provides a classical nonlinear example in which the long-time dynamics depend strongly on the initial data. We consider
$$
W(u)=\frac14(u^{2}-1)^{2}, \quad u\in\mathbb{R}.
$$
This potential has exactly three critical points, $u\in\{-1,0,1\}$, with two symmetric global minima at $u=\pm 1$ and a local maximum at $u=0$. The corresponding damped equation,
$$
\ddot{u}(t)+a\dot{u}(t)+u(u^{2}-1)=0,
$$
has as the set of equilibrium points given by 
$$
\mathcal{E}^{\ast}=\{(-1,0),(0,0),(1,0)\}.
$$
By the structure of $W$, the equilibria $u=\pm 1$ are asymptotically stable and the equilibrium $u=0$ is unstable. Thus the dynamics are bistable, that is, every solution converges to either $(1,0)$ or $(-1,0)$, depending on the initial condition, and the unstable equilibrium $(0,0)$ serves as a separatrix between the two basins of attraction. This behavior contrasts sharply with the quadratic potential considered previously. In the linear case, the system has a single globally asymptotically stable equilibrium, and all trajectories decay monotonically or oscillate before converging to zero, depending solely on the value of the damping parameter. Here, nonlinearity introduces multiple equilibria and sensitivity to initial conditions, producing a qualitatively different global structure.
	
In the first plot we present the phase portrait for the damped system with a fixed damping coefficient $a = 0.3$. The numerical simulations illustrate the bistable nature of the dynamics, where the asymptotic behavior is determined by the basin of attraction in which the initial data lies. 
Specifically, the trajectories starting at $(-1.5, 0.5)$ and the small perturbation $(-0.01, 0)$ are attracted to the stable equilibrium $(-1, 0)$. Conversely, the trajectories starting at $(1.5, -0.2)$ and $(0.01, 0)$ converge to the symmetric equilibrium $(1, 0)$. 
It is worth noting the sensitivity of the dynamics near the origin. The initial conditions $(\pm 0.01, 0)$ are chosen arbitrarily close to the unstable equilibrium $(0, 0)$. As observed, a minute variation in the sign of the initial position is sufficient to drive the solution into different basins of attraction, highlighting the role of the stable manifold of the origin as a separatrix between the two asymptotic regimes.
	
Meanwhile, the second figure displays the dynamics in the conservative limit ($a = 0$). 
In this regime, the system is Hamiltonian and preserves the total energy 
\[
E(u, \dot{u}) = \frac{1}{2}\dot{u}^2 + W(u). 
\]
Consequently, the asymptotic convergence observed in the damped case is replaced by 
the presence of periodic orbits. 

\medskip

\begin{center}
	\includegraphics[width=\textwidth]{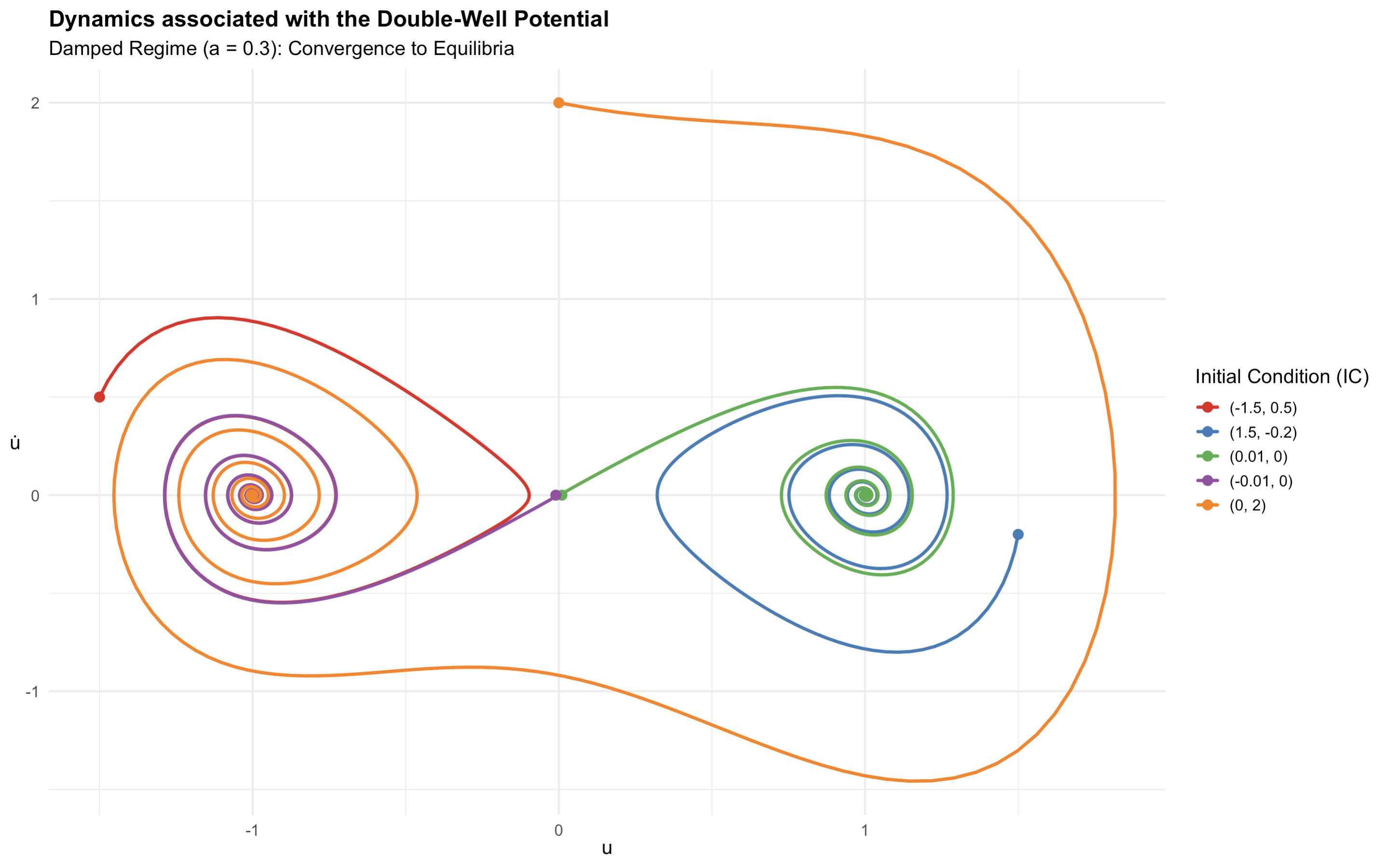}
\end{center}

\vspace{0.5cm}

\begin{center}
	\includegraphics[width=\textwidth]{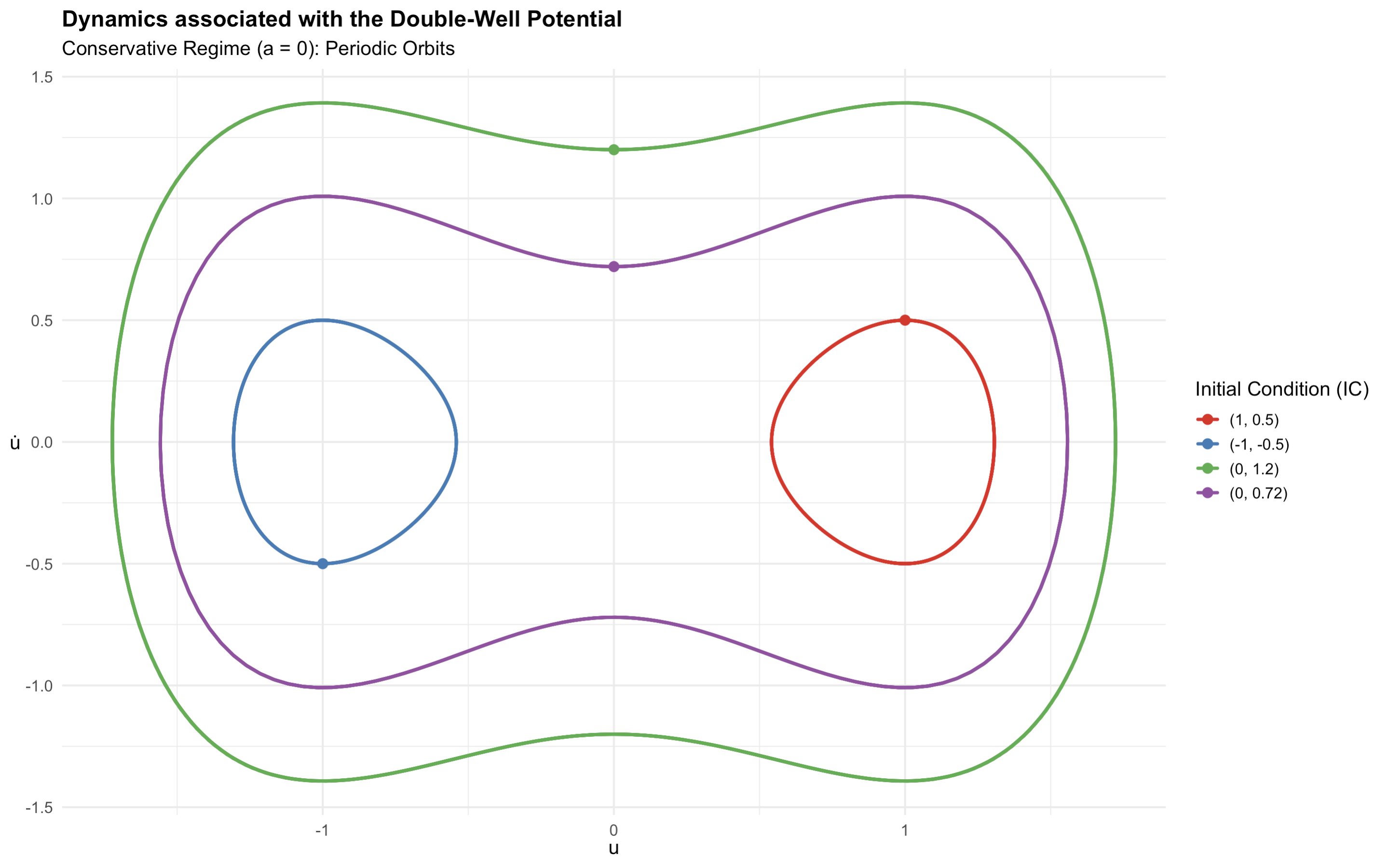}
\end{center}

\medskip

\begin{remark} \rm 
A similar analysis can be carried out for higher-order double-well potentials, such as
$$
W(u)=u^{2}(u^{2}-1)^{2},
$$
which preserve the bistable structure while introducing stronger nonlinearities. The qualitative features remain analogous.
\end{remark}


\subsection{Exponential Potential}

A third example illustrating a qualitatively different nonlinear behavior is provided by the exponential confinement potential
$$
W(u)=\frac12\big(e^{u^{2}}-1\big), \quad u\in\mathbb{R}.
$$
The function $W$ satisfies the assumptions \hyperlink{W1}{$(W_1)$}-\hyperlink{W4}{$(W_4)$}. It possesses a single critical point, namely $u=0$, which is the unique global minimum. Unlike polynomial potentials, $W(u)$ exhibits superquadratic growth, $W(u)\sim e^{u^{2}}/2$ as $|u|\to\infty$, resulting in a rapidly increasing restoring force. The associated damped equation is
$$
\ddot{u}(t)+a\dot{u}(t)+u e^{u^{2}}=0.
$$
The equilibrium set of the corresponding first-order formulation reduces to $\mathcal{E}^{\ast}=\{(0,0)\}$,
which is asymptotically stable. 
 
\medskip

\begin{center}
\includegraphics[width=\textwidth]{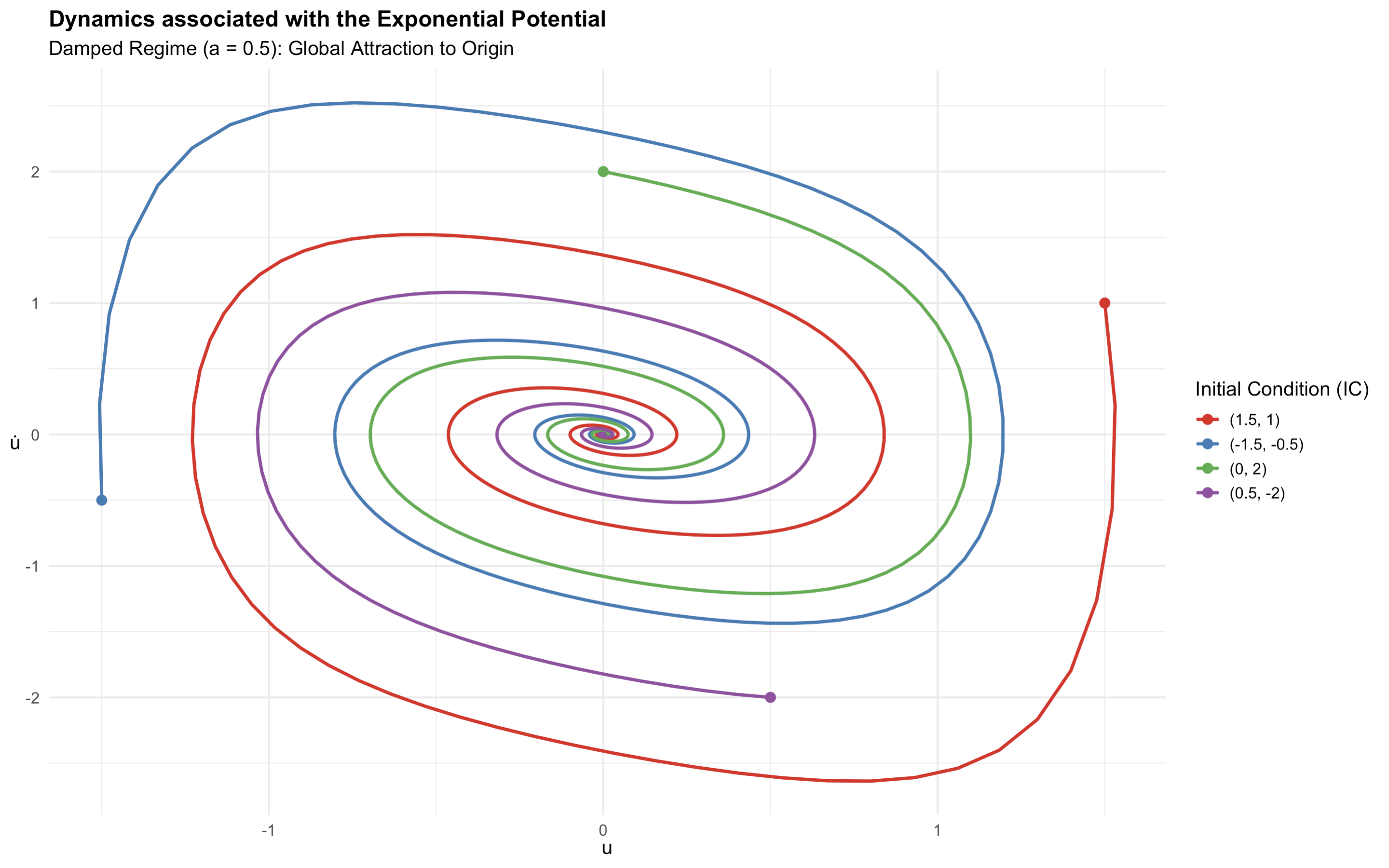}
\end{center}

\vspace{0.5cm}

\begin{center}
\includegraphics[width=\textwidth]{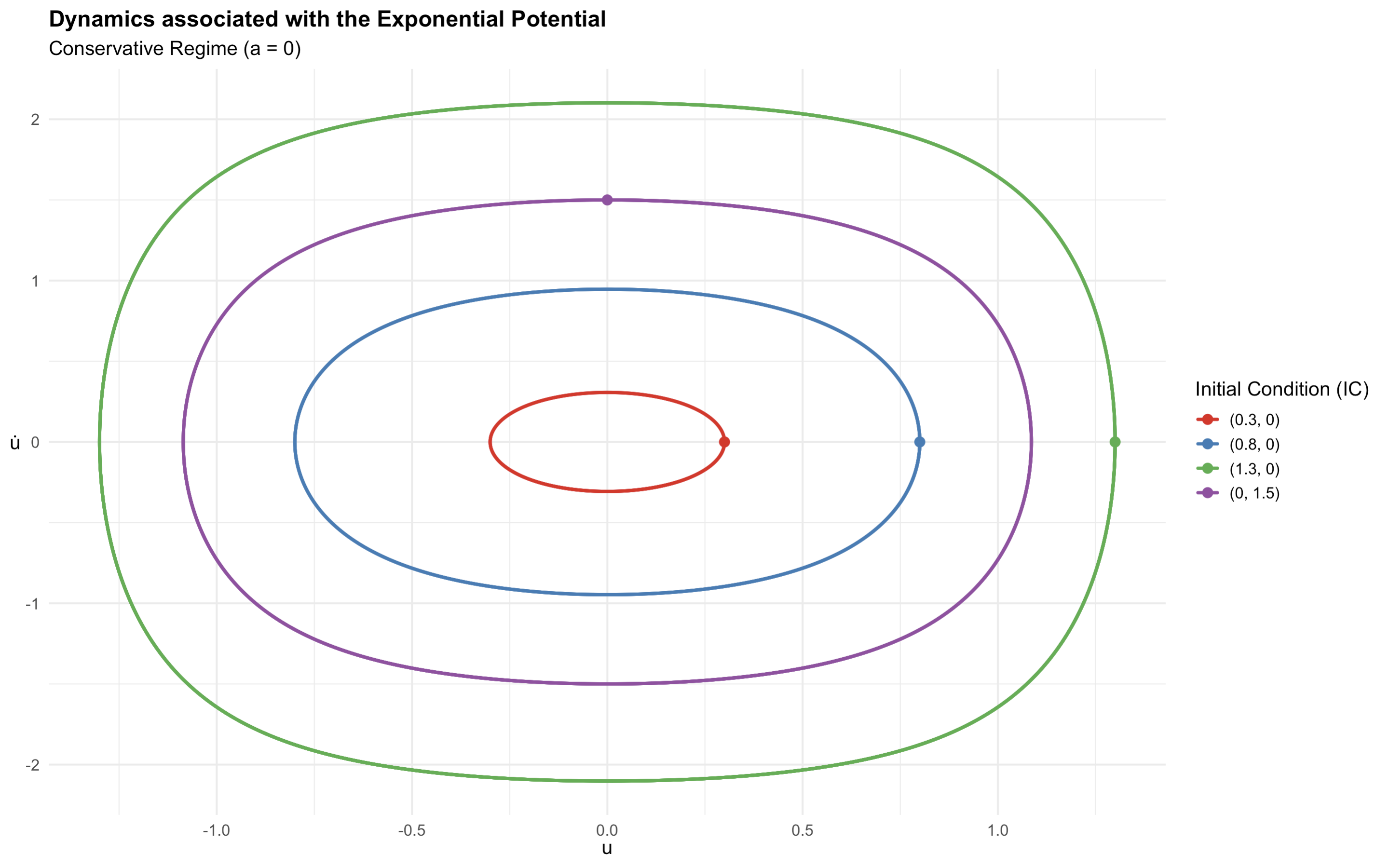}
\end{center}

\medskip

In contrast with the double-well, the exponential potential is strictly convex and possesses a unique critical point at the origin $(0,0)$. 
The numerical simulations demonstrate that the origin acts as a global attractor. Regardless of the initial conditions selected, all trajectories eventually decay to the zero state. 
The choice of $a=0.5$ places the system in an underdamped regime, resulting in spiral orbits as the solutions oscillate with decreasing amplitude around the equilibrium before reaching its resting state. 

On the other hand, the phase portrait in the conservative limit ($a = 0$) shows a confinement effect. 
The phase curves form a family of concentric, closed, periodic orbits surrounding the center at $(0,0)$ 
confirming that all solutions in this regime are bounded and periodic. 

\medskip

\begin{center}
\includegraphics[width=\textwidth]{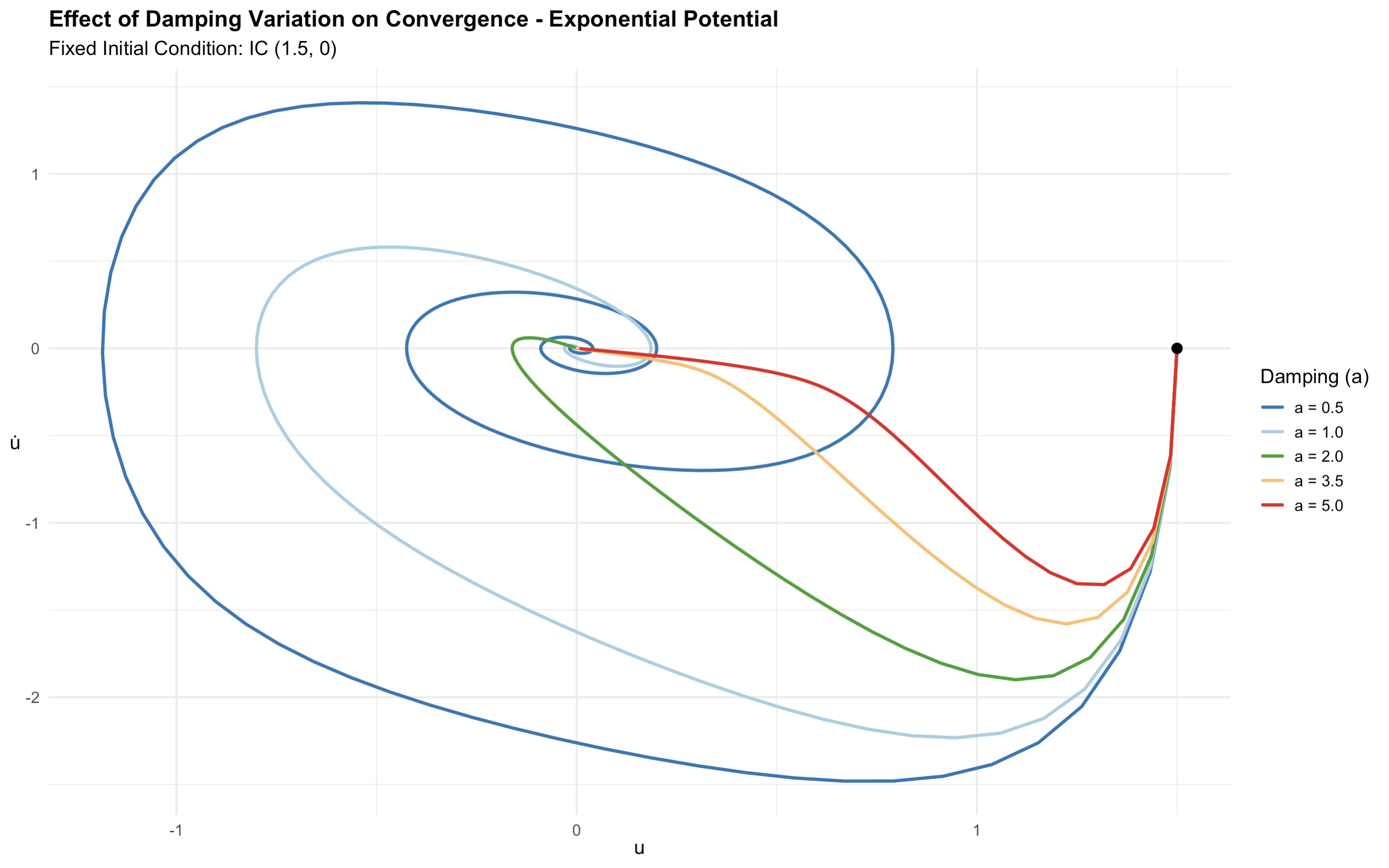}
\end{center}

\medskip

At last, in order to investigate the influence of the dissipation mechanism on the rate of convergence, 
we performed a sensitivity analysis on the damping parameter $a$, 
namely we present the phase portraits for five distinct values of $a \in \{0.5, 1.0, 2.0, 3.5, 5.0\}$, 
keeping the initial condition fixed at $(1.5, 0)$. 

We recall that for the linearized equation \eqref{QP} near the equilibrium, 
the value $a = 2$ represents the critical damping threshold. 
The numerical results for the nonlinear system preserve the qualitative features of the linear analysis, 
allowing us to distinguish three main regimes: 
\begin{enumerate}
\item Underdamped regime ($a < 2$): For the values $a=0.5$ and $a=1.0$, the trajectories exhibit oscillatory decay. As expected, the case $a=0.5$ (dark blue) displays stronger oscillations compared to 
$a=1.0$ (light blue), where the dissipation is sufficient to dampen the motion more rapidly. 
	
\item Critically damped regime ($a = 2$): The trajectory for $a=2.0$ (green) returns to the equilibrium as quickly as possible without oscillating. 
	
\item Overdamped regime ($a > 2$): For $a=3.5$ and $a=5.0$, the strong dissipation prevents the velocity from building up, 
illustrating the effect caused by excessive damping.
\end{enumerate}


\subsection{Implementation}

In this last subsection we present the documentation and scripts containing the code that generated the results in our examples. 

\vspace{0.5cm}

\noindent {\bf Script of the first example: damped harmonic oscillator.}

\begin{rcode}
library(deSolve)
library(ggplot2)
library(dplyr)

# --- 1. System Definition: Quadratic Potential ---
quadratic_potential <- function(t, state, parameters) {
	u <- state[1]
	v <- state[2]
	a <- parameters["a"]
	
	du <- v
	dv <- -a * v - u  # Linear restoring force
	
	list(c(du, dv))
}

# --- 2. PIPELINE A: Varying Damping ---
plot_damping_a <- function(fixed_ic, damping_values, time_seq) {
	
	sim_data <- lapply(damping_values, function(val_a) {
		out <- ode(y = fixed_ic, times = time_seq, func = quadratic_potential, parms = c(a = val_a))
		label_str <- paste("a =", sprintf("%.1f", val_a))
		as.data.frame(out) %>% mutate(traj = label_str)
	}) %>%
	bind_rows()
	
	sim_data$traj <- factor(sim_data$traj, levels = unique(sim_data$traj))
	
	# PALETTE.
	# Blue(a=0) -> Cyan -> Green -> Yellow(a=2) -> Orange -> Red(a=6)
	bright_colors <- c("#0072B2", "#56B4E9", "#009E73", "#F0E442", "#E69F00", "#D55E00")
	col_map <- bright_colors[1:length(unique(sim_data$traj))]
	
	p <- ggplot(sim_data, aes(x = u, y = v, color = traj, group = traj)) +
	geom_path(linewidth = 1.1, lineend = "round") + 
	geom_point(data = filter(sim_data, time == 0), size = 3, color = "black") + 
	scale_color_manual(values = col_map) +
	labs(
	title = "Dynamics associated with the Quadratic Potential - Damping Variation",
	subtitle = paste0("Fixed Initial Condition: (", fixed_ic['u'], ", ", fixed_ic['v'], ")"),
	x = expression(u), y = expression(dot(u)), color = "Damping (a)"
	) +
	theme_minimal(base_size = 14) +
	theme(
	plot.title = element_text(face = "bold"),
	legend.position = "right",
	axis.title.y = element_text(angle = 0, vjust = 0.5)
	)
	return(p)
}

# --- 3. PIPELINE B: Varying Initial Conditions ---
plot_phase_portrait <- function(damping_a, ic_list, time_seq, subtitle_text) {
	
	sim_data <- lapply(ic_list, function(ic) {
		out <- ode(y = ic, times = time_seq, func = quadratic_potential, parms = c(a = damping_a))
		label_str <- sprintf("(%s, %s)", round(ic['u'], 2), round(ic['v'], 2))
		as.data.frame(out) %>% mutate(traj = label_str)
	}) %>%
	bind_rows()
	
	sim_data$traj <- factor(sim_data$traj, levels = unique(sim_data$traj))
	
	# PALETTE.
	dark2_colors <- c("#1B9E77", "#D95F02", "#7570B3", "#E7298A", "#66A61E", "#E6AB02")
	col_map <- dark2_colors[1:length(unique(sim_data$traj))]
	
	p <- ggplot(sim_data, aes(x = u, y = v, color = traj, group = traj)) +
	geom_path(linewidth = 1.1, lineend = "round") + 
	geom_point(data = filter(sim_data, time == 0), size = 3) + 
	scale_color_manual(values = col_map) +
	labs(
	title = "Dynamics associated with the Quadratic Potential considering different Initial Conditions",
	subtitle = paste(subtitle_text, "| a =", damping_a),
	x = expression(u), y = expression(dot(u)), color = "IC"
	) +
	theme_minimal(base_size = 14) +
	theme(
	plot.title = element_text(face = "bold"),
	legend.position = "right",
	axis.title.y = element_text(angle = 0, vjust = 0.5)
	)
	return(p)
}

# --- 4. EXECUTION ---

# Plot varying the damping coefficient.

times_sens <- seq(0, 15, by = 0.05)
ic_fixed   <- c(u = 2, v = 0)

# 6 values covering all regimes:
# Conservative(0), Under(0.5, 1.0), Critical(2.0), Over(3.5, 6.0)
a_values   <- c(0, 0.5, 1.0, 2.0, 3.5, 6.0)

plot1 <- plot_damping_a(ic_fixed, a_values, times_sens)

print(plot1)

# Plot periodic orbits in the conservative case (a = 0).

times_pp <- seq(0, 25, by = 0.05)
ics_list <- list(
c(u = -0.5, v = -1),
c(u = 2, v = 0),
c(u = -2, v = -1),
c(u = 0, v = 2.5)
)

plot2 <- plot_phase_portrait(0, ics_list, times_pp, "Conservative Regime - Periodic Orbits")

print(plot2)
\end{rcode}


\vspace{0.5cm}

\noindent {\bf Script of the second example: double-well potential.}

\begin{rcode}
library(deSolve)
library(ggplot2)
library(dplyr)

# --- 1. Define the system with double-well potential ---
double_well <- function(t, state, parameters) {
	u <- state[1]
	v <- state[2]
	a <- parameters["a"] 
	
	du <- v
	dv <- -a * v - u * (u^2 - 1)
	
	list(c(du, dv))
}


# 2. Function that will run the simulation for the damped ODE with the double-well potential
# for a given value of the damping coefficient and a list of initial conditions
# and it will return the plot containing the phase portraits.

run_simulation <- function(damping_a, initial_conditions_list, time_seq, plot_subtitle) {
	
	# Run ODE.
	simulation_data <- lapply(initial_conditions_list, function(ic) {
		out <- ode(y = ic, times = time_seq, func = double_well, parms = c(a = damping_a))
		label_str <- sprintf("(%s, %s)", round(ic['u'], 2), round(ic['v'], 2))
		as.data.frame(out) %>% mutate(traj = label_str)
	}) %>%
	bind_rows()
	
	# Reorder Factors.
	unique_labels <- unique(simulation_data$traj)
	simulation_data$traj <- factor(simulation_data$traj, levels = unique_labels)
	
	# Define the vector with colors.
	my_colors <- c("#E41A1C", "#377EB8", "#4DAF4A", "#984EA3", "#FF7F00", "#A65628", "#F781BF")
	color_mapping <- my_colors[1:length(unique_labels)]
	
	# Plotting.
	p <- ggplot(simulation_data, aes(x = u, y = v, color = traj, group = traj)) +
	geom_path(linewidth = 1.1, lineend = "round") + 
	geom_point(data = filter(simulation_data, time == 0), size = 3) + 
	scale_color_manual(values = color_mapping) +
	labs(
	title = "Dynamics associated with the Double-Well Potential",
	subtitle = plot_subtitle,
	x = expression(u),       
	y = expression(dot(u)), 
	color = "Initial Condition (IC)"
	) +
	theme_minimal(base_size = 14) +
	theme(
	plot.title = element_text(face = "bold", size = 16),
	legend.position = "right",
	legend.text = element_text(size = 11),
	axis.title.y = element_text(angle = 0, vjust = 0.5) 
	)
	
	return(p)
}

# --- 3. Scenarios: damped and conservative ---

# Damped Case.
times_damped <- seq(0, 40, by = 0.05)

inits_damped_list <- list(
c(u = -1.5, v = 0.5),
c(u = 1.5,  v = -0.2),
c(u = 0.01, v = 0),   
c(u = -0.01, v = 0),  
c(u = 0,    v = 2.0)  
)

plot_damped <- run_simulation(0.3, inits_damped_list, times_damped, "Damped Regime (a = 0.3): Convergence to Equilibria")

print(plot_damped)

# Conservative Case.
times_conserv <- seq(0, 30, by = 0.05)
inits_conserv_list <- list(
c(u = 1, v = 0.5),    
c(u = -1, v = -0.5),  
c(u = 0, v = 1.2),    
c(u = 0, v = 0.72)    
)

plot_conserv <- run_simulation(0, inits_conserv_list, times_conserv, "Conservative Regime (a = 0): Periodic Orbits")

print(plot_conserv)
\end{rcode}


\vspace{0.5cm}

\noindent {\bf Script of the third example: exponential potential.}

\begin{rcode}
library(deSolve)
library(ggplot2)
library(dplyr)

# --- 1. System Definition: Exponential Potential ---
# Potential: W(u) = 0.5 * (exp(u^2) - 1)
# Gradient:  W'(u) = u * exp(u^2)

exp_potential <- function(t, state, parameters) {
	u <- state[1]
	v <- state[2]
	a <- parameters["a"] 
	
	du <- v
	# The restoring force is now -u * exp(u^2)
	dv <- -a * v - u * exp(u^2)
	
	list(c(du, dv))
}

# --- 2. Function for running simulations and plotting the phase portraits. ---
run_simulation <- function(damping_a, initial_conditions_list, time_seq, plot_subtitle) {
	
	simulation_data <- lapply(initial_conditions_list, function(ic) {
		out <- ode(y = ic, times = time_seq, func = exp_potential, parms = c(a = damping_a))
		# Label formatting
		label_str <- sprintf("(%s, %s)", round(ic['u'], 2), round(ic['v'], 2))
		as.data.frame(out) %>% mutate(traj = label_str)
	}) %>%
	bind_rows()
	
	# Factor reordering for Legend
	unique_labels <- unique(simulation_data$traj)
	simulation_data$traj <- factor(simulation_data$traj, levels = unique_labels)
	
	# Palette
	my_colors <- c("#E41A1C", "#377EB8", "#4DAF4A", "#984EA3", "#FF7F00", "#A65628")
	color_mapping <- my_colors[1:length(unique_labels)]
	
	# Plotting
	p <- ggplot(simulation_data, aes(x = u, y = v, color = traj, group = traj)) +
	geom_path(linewidth = 1.1, lineend = "round") + 
	geom_point(data = filter(simulation_data, time == 0), size = 3) + 
	scale_color_manual(values = color_mapping) +
	labs(
	title = "Dynamics associated with the Exponential Potential",
	subtitle = plot_subtitle,
	x = expression(u),       
	y = expression(dot(u)), 
	color = "Initial Condition (IC)" 
	) +
	theme_minimal(base_size = 14) +
	theme(
	plot.title = element_text(face = "bold", size = 16),
	legend.position = "right",
	axis.title.y = element_text(angle = 0, vjust = 0.5) 
	)
	
	return(p)
}

# --- 3. Scenario A: Damped Case (a = 0.5) ---

times_damped <- seq(0, 25, by = 0.05)

inits_damped_list <- list(
c(u = 1.5, v = 1.0),
c(u = -1.5, v = -0.5),
c(u = 0, v = 2.0),
c(u = 0.5, v = -2.0) 
)

plot_damped <- run_simulation(
damping_a = 0.5, 
initial_conditions_list = inits_damped_list, 
time_seq = times_damped, 
plot_subtitle = "Damped Regime (a = 0.5): Global Attraction to Origin"
)

print(plot_damped)

# --- 4. Scenario B: Conservative Case (a = 0) ---

times_conserv <- seq(0, 15, by = 0.01)

inits_conserv_list <- list(
c(u = 0.3, v = 0),
c(u = 0.8, v = 0),
c(u = 1.3, v = 0),
c(u = 0,   v = 1.5)
)

plot_conserv <- run_simulation(
damping_a = 0, 
initial_conditions_list = inits_conserv_list, 
time_seq = times_conserv, 
plot_subtitle = "Conservative Regime (a = 0)"
)

print(plot_conserv)
\end{rcode}


\medskip

\noindent Script: varying the damping coefficient within the exponential potential. 

\medskip

\begin{rcode}
library(deSolve)
library(ggplot2)
library(dplyr)

# --- 1. System Definition ---
exp_potential <- function(t, state, parameters) {
	u <- state[1]
	v <- state[2]
	a <- parameters["a"] 
	
	du <- v
	dv <- -a * v - u * exp(u^2)
	
	list(c(du, dv))
}

# --- 2. Function for running simulations with different values of the damping. ---
# It returns the phase portraits, where each curve corresponds to a value of the damping.

run_damping_comparison <- function(fixed_ic, damping_values, time_seq) {
	
	simulation_data <- lapply(damping_values, function(val_a) {
		out <- ode(y = fixed_ic, times = time_seq, func = exp_potential, parms = c(a = val_a))
		
		# Label formatting
		label_str <- paste("a =", sprintf("%.1f", val_a))
		as.data.frame(out) %>% mutate(traj = label_str)
	}) %>%
	bind_rows()
	
	# Ensure the legend order matches the numeric order of 'a'
	unique_labels <- unique(simulation_data$traj)
	simulation_data$traj <- factor(simulation_data$traj, levels = unique_labels)
	
	# Custom Color Palette.
	# Blue/Cyan for Underdamped -> Green for Critical -> Orange/Red for Overdamped.
	my_colors <- c("#1f78b4", # Strong Blue (a=0.5)
	"#a6cee3", # Light Blue (a=1.0)
	"#33a02c", # Green      (a=2.0 - Critical)
	"#fdbf6f", # Orange     (a=3.5)
	"#e31a1c") # Red        (a=5.0)
	
	color_mapping <- my_colors[1:length(unique_labels)]
	
	p <- ggplot(simulation_data, aes(x = u, y = v, color = traj, group = traj)) +
	geom_path(linewidth = 1.1, lineend = "round") + 
	geom_point(data = filter(simulation_data, time == 0), size = 3, color = "black") + 
	scale_color_manual(values = color_mapping) +
	labs(
	title = "Effect of Damping Variation on Convergence - Exponential Potential",
	subtitle = paste0("Fixed Initial Condition: IC (", fixed_ic['u'], ", ", fixed_ic['v'], ")"),
	x = expression(u),       
	y = expression(dot(u)), 
	color = "Damping (a)" 
	) +
	theme_minimal(base_size = 14) +
	theme(
	plot.title = element_text(face = "bold", size = 16),
	legend.position = "right",
	axis.title.y = element_text(angle = 0, vjust = 0.5) 
	)
	
	return(p)
}

# --- 3. Execution ---

times <- seq(0, 20, by = 0.05)
fixed_ic <- c(u = 1.5, v = 0)

# Vector with five positive values for the damping.
damping_values <- c(0.5, 1.0, 2.0, 3.5, 5.0)

plot_damping <- run_damping_comparison(fixed_ic, damping_values, times)

print(plot_damping)
\end{rcode}


\end{document}